\documentclass[reqno,12pt]{amsart}%[draft]
\usepackage{amscd,amsfonts,amsmath,amssymb,amstext,amsthm}

\usepackage[utf8]{inputenc}
\usepackage[T2A]{fontenc}
\usepackage[russian,english,]{babel}
\usepackage{cite}
\usepackage[unicode]{hyperref}
\usepackage{color}
\usepackage{xcolor}
\usepackage{comment}
\usepackage{tikz}
\usepackage{wrapfig}
\usepackage[normalem]{ulem}
\addtocontents{toc}{\setcounter{tocdepth}{4}} % Set depth to 1

\makeatletter
\def\@settitle{\begin{center}%
    \baselineskip14\p@\relax
    \bfseries
    \MakeUppercase{\@title}
  \end{center}
}
\makeatother

\newtheorem{theorem}{Theorem}[section]
\newtheorem{lemma}{\it Lemma}[section]
\newtheorem{proposition}{Proposition}[section]
\newtheorem{corollary}{\it Corollary}[section]
\theoremstyle{remark}

\newtheorem{remark}{Remark}[section]
\newtheorem{example}{\it Example}[section]
\theoremstyle{definition}
\newtheorem{assert}{Assertion}
\newtheorem{definition}{\it Definiton}[section]
\numberwithin{equation}{section}

\def\Cn*{{\C^n}^*}
\def\CN*{{\C^N}^*}

\def\C{{\mathbb C}}
\def\Q{{\mathbb Q}}

\def\P{{\mathbb P}}

\def\Y{{\mathbb Y}}

\def\R{{\mathbb R}}
\def\Z{{\mathbb Z}}

\def\vol{{\rm vol}}
\def\E{\:{\rm e}}
\def\re{{\rm Re\:}}
\def\im{{\rm Im\:}}

\def\conv{{\rm conv}}

\begin{document}
\title{Around the ``Fundamental Theorem of Algebra''
}
{
\author{B. Kazarnovskii
}
\address {\noindent
Moscow Institute of Physics and
Technology (National Research University),
Higher School of Contemporary Mathematics
\newline
{\it kazbori@gmail.com}.}
}
%\thanks {%\noindent
%%\support {\noindent
%%\newline
%This research was carried out at the Higher School of Contemporary Mathematics of Moscow
%Institute of Physics and Technology,
%with the support of the Ministry of Science and Higher Education
%of the Russian Federation,
%project no. SMG - 2024 -0048.
%}
\thanks {%\noindent
The research is supported by MSHE RF GZ project}
\keywords{Fundamental Theorem of Algebra, random polynomial, mean number of zeros, Newton ellipsoid, exponential sum}

\begin{abstract}
The Fundamental Theorem of Algebra (FTA) asserts that every complex polynomial has as many complex roots, counted with multiplicities, as its degree.
A probabilistic analogue of this theorem for real roots of real polynomials—sometimes referred to as the Kac theorem—was
found between 1938 and 1943 by J. Littlewood, A. Offord, and M. Kac.
In this paper, we present several more versions of FTA:
Kac type FTA for Laurent polynomials in one and many variables,
Kac type FTA for polynomials
on complex reductive groups
 arising in the context of compact group representations
 (similar to Laurent polynomials arising in torus representation theory),
and FTA for exponential sums in one and many variables.
In the case of Laurent polynomials, the result, even in the one-dimensional case, is unexpected:
most of the zeros of a real Laurent polynomial are real.
\end{abstract}
\maketitle
\tableofcontents
\section{Introduction}\label{Intro}
A probabilistic refinement of the Fundamental Theorem of Algebra (FTA),
concerning the real roots of random real polynomials and sometimes referred to as
\emph{Kac theorem},
was %developed
found between 1938 and 1943 by J.~Littlewood, A.~Offord, and M.~Kac;
see~\cite{Littl1,Littl2,Littl3,KA}.
The underlying observation is elementary:
for any integer $m$, the probability that a root of a real polynomial of degree~$m$
is real is strictly positive.
In particular, for $m=1$ this probability equals~$1$.
\begin{theorem}[Kac]\label{thmKac}
Let $P$ be a random real polynomial of degree~$m$ in one variable whose coefficients
are independent Gaussian random variables with zero mean and unit variance.
Then, as $m\to\infty$, the expected number of real roots of $P$ is asymptotic to
$\frac{2}{\pi}\log m$.
Equivalently,
\[
\mathcal P(m)\asymp \frac{2}{\pi}\frac{\log m}{m},
\]
where $\mathcal P(m)$ denotes the probability that a root of $P$ is real.
\end{theorem}
\begin{remark}
For refined results on the distribution of the number of real roots of random real polynomials,
we refer to~\cite{EK} and the references therein.
\end{remark}
\begin{remark}
If the expected value of the $k$th coefficient is taken to be $\binom{m}{k}$,
then in Kac's theorem the logarithmic term $\log m$ is replaced by $\sqrt m$.
\end{remark}
In the present paper we establish Kac-type results for Laurent polynomials in one and several variables,
as well as for polynomials on complex reductive groups arising naturally in the representation theory
of compact groups.
The latter class may be viewed as an analogue of Laurent polynomials
associated with torus representations.
A notable feature of this setting is that,
in contrast with the classical polynomial case,
the probability that a root is real converges to a nonzero limit as the degree tends to infinity.
For example, in the one-variable Laurent case this limit equals $1/\sqrt3$;
see Corollary~\ref{corL1}(c).

In \S\ref{exp} we further formulate and prove an analogue of the FTA for systems of equations
defined by exponential sums in $\C^n$.
Since such systems have infinitely many solutions,
the corresponding statement takes the form of an explicit density formula
for the zero set.
This result is analogous to the Bernstein--Kushnirenko--Khovanskii formula
for polynomial systems.
The material of \S\ref{exp} is independent of \S\ref{Laur} and \S\ref{group}.
\section{Kac type theorem for Laurent polynomials}\label{Laur}
\subsection{Laurent polynomials of one variable}\label{Laur1}
Recall that a Laurent polynomial is a function on $\C\setminus0$
of the form
\[
P(z)=\sum_{m\in\Lambda} a_m z^m,
\]
where $\Lambda\subset\Z$ is a finite set, called the \emph{spectrum} of~$P$.

\begin{definition}\label{dfLaur}
\begin{enumerate}
\item
A Laurent polynomial $P$ is called \emph{real} if its restriction to the unit circle
$S\subset\C$ is real-valued.

\item
A root of $P$ lying on $S$ is called a \emph{real root} of the Laurent polynomial.
\end{enumerate}
\end{definition}
The following assertions are immediate consequences of Definition~\ref{dfLaur}.
\begin{lemma}\label{lmL1}
\begin{enumerate}
\item
A Laurent polynomial $P(z)=\sum a_k z^k$ is real if and only if
\[
a_k=\overline{a_{-k}}\qquad \text{for all }k\in\Z.
\]
In particular, the spectrum of a real Laurent polynomial is centrally symmetric.

\item
The set of roots of a real Laurent polynomial is invariant under the involution
$z\mapsto\bar z^{-1}$.

\item
Restriction to~$S$ defines an isomorphism between the space of real Laurent polynomials
with spectrum~$\Lambda$ and the space of trigonometric polynomials with the same spectrum,
namely functions of the form
\begin{equation}\label{tr}
\sum_{k\in\Lambda}\bigl(\alpha_k\cos(k\theta)+\beta_k\sin(k\theta)\bigr),
\qquad \alpha_k,\beta_k\in\R.
\end{equation}
\end{enumerate}
\end{lemma}
\begin{definition}[Random real Laurent polynomials]\label{dfTrig}
Let ${\rm Trig}(\Lambda)$ denote the space of trigonometric polynomials of the form~\eqref{tr}.
We regard ${\rm Trig}(\Lambda)$ as a finite-dimensional subspace of $L_2(S)$,
endowed with the inner product
\[
(f,g)=\int_S f(\theta)g(\theta)\,d\theta.
\]
Random real Laurent polynomials with spectrum~$\Lambda$
are defined by equipping ${\rm Trig}(\Lambda)$ with the Gaussian probability measure
associated with this $L_2$-metric.
\end{definition}
\begin{theorem}\label{thmL2}
Let $f_\Lambda$ be a random real Laurent polynomial with spectrum~$\Lambda$.
Then:

\begin{enumerate}
\item[(a)]
The expected number of real roots of $f_\Lambda$ equals
\[
2\cdot\sqrt{\frac{1}{\#\Lambda}\sum_{\lambda\in\Lambda}\lambda^2},
\]
where $\#\Lambda$ denotes the cardinality of~$\Lambda$.

\item[(b)]
The probability $\mathcal P(\Lambda)$ that a root of $f_\Lambda$ is real is given by
\[
\mathcal P(\Lambda)
=
\frac{1}{\deg(f_\Lambda)}
\sqrt{\frac{1}{\#\Lambda}\sum_{\lambda\in\Lambda}\lambda^2},
\]
where $\deg(f_\Lambda)=\max_{\lambda\in\Lambda}|\lambda|$.
\end{enumerate}
\end{theorem}
\begin{corollary}\label{corL1}
Let $\Lambda=\{-m,\ldots,0,\ldots,m\}$, and let $f_m$ be a random real Laurent polynomial
with spectrum~$\Lambda$, i.e.\ a random real Laurent polynomial of degree~$m$.
Then:
\begin{enumerate}
\item[(a)]
The expected number of real roots of $f_m$ equals
\[
2\sqrt{\frac{m(m+1)}{3}}.
\]

\item[(b)]
The probability $\mathcal P(f_m)$ that a root of $f_m$ is real equals
\[
\sqrt{\frac{m+1}{3m}}.
\]

\item[(c)]
\[
\lim_{m\to\infty}\mathcal P(f_m)=\frac{1}{\sqrt3}.
\]
\end{enumerate}
\end{corollary}

\begin{proof}
This follows from the identity
\[
\sum_{k=1}^m k^2=\frac{m(m+1)(2m+1)}{6}.
\]
\end{proof}
\begin{remark}
Since $1/\sqrt3>1/2$, Corollary~\ref{corL1}(c) implies that
most zeros of a random real Laurent polynomial are real.
\end{remark}
\begin{remark}
The statements of Corollary~\ref{corL1} are equivalent to computing
the expected number of zeros of trigonometric polynomials of fixed degree on the circle;
see~\cite{ADG}.
\end{remark}
The proof of Theorem~\ref{thmL2} relies on a classical Crofton-type formula;
see, for example,~\cite{Sa}.

\begin{proposition}\label{prCr}
Let $\mu$ be the standard Gaussian measure on $\R^n$, given by
\[
\mu(U)=(2\pi)^{-n/2}\int_U e^{-|x|^2/2}\,dx.
\]
For $0\neq\xi\in\R^n$, set
\[
Z(\xi)=\{x\in\R^n\mid (x,\xi)=0\}.
\]
Let $K$ be a compact curve in the unit sphere $S^{n-1}\subset\R^n$.
Then
\[
\int_{\R^n}\#(K\cap Z(\xi))\,d\mu(\xi)
=
\frac{1}{\pi}\,\mathrm{length}(K).
\]
\end{proposition}
 To apply Proposition~\ref{prCr}, let $S\subset\C$ denote the unit circle.

\begin{definition}\label{dfTheta1}
For $s\in S$, consider the linear functional $\tilde s$ on
${\rm Trig}(\Lambda)$ defined by $\tilde s(f)=f(s)$.
Using the $L_2(S)$ inner product
\[
(f,g)=\int_S f(s)g(s)\,ds,
\]
we define a mapping $\Theta\colon S\to{\rm Trig}(\Lambda)$ by the condition
\[
(f,\Theta(s))=\tilde s(f)
\qquad\text{for all }f\in{\rm Trig}(\Lambda).
\]
\end{definition}
The expected number of real roots of a random real Laurent polynomial
with spectrum~$\Lambda$ equals to
\begin{equation}\label{eqForCr}
\int_{{\rm Trig}(\Lambda)}
\#\bigl(\Theta(S)\cap Z(\xi)\bigr)\,d\mu(\xi),
\end{equation}
where $Z(\xi)=\{x\in{\rm Trig}(\Lambda)\mid (x,\xi)=0\}$.
\begin{corollary}\label{cor1}
Let $K_\Lambda$ denote the image of $S$ under the map
$s\mapsto\Theta(s)/|\Theta(s)|$.
Then the expected number of real roots equals
\[
\frac{1}{\pi}\,\mathrm{length}(K_\Lambda).
\]
\end{corollary}
\begin{proof}
This is an immediate consequence of Proposition~\ref{prCr}.
\end{proof}
Consequently, Theorem~\ref{thmL2}(a) is equivalent to the identity
\begin{equation}\label{eqKTheta}
\mathrm{length}(K_\Lambda)
=
2\pi\sqrt{\frac{1}{\#\Lambda}\sum_{\lambda\in\Lambda}\lambda^2}.
\end{equation}

To compute the length of the curve $K_\Lambda$ we require the following auxiliary statement,
which will also be used in \S\S\ref{Laur41}, \ref{group3} in the proof of multidimensional analogues of Theorem~\ref{thmL2}.

\begin{lemma}\label{lmLin}
Let
\begin{itemize}
\item $X$ be a homogeneous space of a compact Lie group $K$, equipped with a $K$-invariant Riemannian metric $g$ normalized by $\int_X dg = 1$;
\item $V\subset C^\infty(X)$ be a finite-dimensional $K$-invariant subspace endowed with a $K$-invariant inner product
$$
(\phi,\psi)=\int_X\phi \psi\:dg
$$
\end{itemize}
For $x\in X$ define $F_x\in V$ by
\begin{equation}\label{eqFx}
F_x=\sum_{i=1}^N f_i(x)\,f_i,
\end{equation}
where $\{f_1,\dots,f_N\}$ is an orthonormal basis of $V$.
Then:
\begin{enumerate}
\item For all $x\in X$ and $f\in V$ one has $f(x)=(f,F_x)$.
\item The map $\kappa\colon X\to V$, $\kappa(x)=\frac{1}{\sqrt N}F_x$, is independent of the choice of the orthonormal basis $\{f_i\}$.
\item The image $\kappa(X)$ lies on the unit sphere in $V$, i.e.\ $(F_x,F_x)=N$ for all $x\in X$.
\end{enumerate}
\end{lemma}

\begin{proof}
(1) For each $k\le N$ we have
\[
(f_k,F_x)=\sum_{i=1}^N f_i(x)(f_k,f_i)=f_k(x),
\]
which implies the claim by linearity.

(2) Assertion~(1) shows that $F_x$ is characterized by the identities $(f,F_x)=f(x)$ for all $f\in V$, and hence does not depend on the choice of the orthonormal basis.

(3) Since the action of $K$ on $V$ is orthogonal, the function $x\mapsto |F_x|^2$ is $K$-invariant and therefore constant on $X$.
Let us denote  this constant by $r^2$.
Integrating  $r^2$ over~$X$ using \eqref{eqFx} we obtain
\[
r^2=\int_X (F_x,F_x)\,dg
   =\sum_{i=1}^N\int_X f_i(x)^2\,dg
   =N,
\]
where we used the orthonormality of $\{f_i\}$.
Thus $r=\sqrt N$, proving~(3).
\end{proof}
\begin{corollary}\label{corLin}
Let $\Theta\colon S\to\mathrm{Trig}(\Lambda)$ be the map from Definition~\ref{dfTheta1}.
Choose the orthonormal basis $\{f_\lambda\}_{\lambda\in\Lambda}$ of $\mathrm{Trig}(\Lambda)$ defined by
\[
f_\lambda=
\begin{cases}
\dfrac{1}{\sqrt\pi}\cos(\lambda\theta), & \lambda>0,\\[0.6em]
\dfrac{1}{\sqrt\pi}\sin(\lambda\theta), & \lambda<0,\\[0.6em]
\dfrac{1}{\sqrt{2\pi}}, & \lambda=0\in\Lambda.
\end{cases}
\]
Then for all $s\in S$:
\begin{enumerate}
\item
\[
\Theta(s)=
f_0+\frac{1}{\sqrt\pi}
\sum_{0<\lambda\in\Lambda}
\bigl(\cos(\lambda s)\,f_\lambda+\sin(\lambda s)\,f_{-\lambda}\bigr).
\]
\item $|\Theta(s)|^2=\#\Lambda$.
\item
\[
\Theta'(s)=\frac{1}{\sqrt\pi}
\sum_{0<\lambda\in\Lambda}
\bigl(-\lambda\sin(\lambda s)\,f_\lambda
      +\lambda\cos(\lambda s)\,f_{-\lambda}\bigr).
\]
\item $|\Theta'(s)|^2=\sum_{\lambda\in\Lambda}\lambda^2$.
\end{enumerate}
\end{corollary}

\begin{proof}
We regard $S$ as a homogeneous space under its own action and set $V=\mathrm{Trig}(\Lambda)$.
By definition one has $\Theta(s)=F_s$, so assertions~(1) and~(2) follow directly from Lemma~\ref{lmLin}.
Assertions~(3) and~(4) are obtained by differentiation of~(1) and the identity $\cos^2+\sin^2=1$.
\end{proof}
Since $|\Theta'(s)|^2=\sum_{\lambda\in\Lambda}\lambda^2$ and $|\Theta(s)|^2=\#\Lambda$, we obtain
\[
\mathrm{length}(K_\Lambda)
=\int_S\sqrt{\frac{1}{\#\Lambda}\sum_{\lambda\in\Lambda}\lambda^2}\,ds
=2\pi\sqrt{\frac{1}{\#\Lambda}\sum_{\lambda\in\Lambda}\lambda^2}.
\]
This proves equality~\eqref{eqKTheta} and hence completes the proof of Theorem~\ref{thmL2}.
\subsection{The Bernstein--Kushnirenko--Khovanskii theorem}\label{Laur2}
We now consider systems of multivariate Laurent polynomials.
Let $\Lambda\subset\Z^n$ be a finite set.
A Laurent polynomial in $n$ variables with spectrum $\Lambda$ is a function
\[
f(z)=\sum_{\lambda\in\Lambda} a_\lambda z^\lambda
\]
on the complex algebraic torus $(\C\setminus0)^n$,
where for $\lambda=(\lambda_1,\ldots,\lambda_n)\in\Z^n$ we write
$z^\lambda=z_1^{\lambda_1}\cdots z_n^{\lambda_n}$.
The set $\Lambda$ is called the \emph{spectrum} of $f$,
and its convex hull
\[
\Delta={\rm conv}(\Lambda)\subset\R^n
\]
is the \emph{Newton polyhedron} of $f$.

The following result, due to Bernstein, Kushnirenko, and Khovanskii,
may be regarded as a multidimensional analogue of the Fundamental Theorem of Algebra;
see, for example,~\cite{EKK}.
\begin{theorem}[BKK]\label{thmBKK}
Let $F=(f_1,\ldots,f_n)$ be a system of Laurent polynomials in $n$ variables,
with spectra $\Lambda_1,\ldots,\Lambda_n$ and Newton polyhedra
$\Delta_i={\rm conv}(\Lambda_i)$.
Let $N(F)$ denote the number of isolated roots of the system
$f_1=\cdots=f_n=0$ in $(\C\setminus0)^n$.
Then, for a generic choice of the coefficients,
all roots are isolated and
\begin{equation}\label{eqBKK}
N(F)=n!\,\vol(\Delta_1,\ldots,\Delta_n),
\end{equation}
where $\vol(\Delta_1,\ldots,\Delta_n)$ denotes the mixed volume of the polyhedra
$\Delta_1,\ldots,\Delta_n$.

More precisely, in the space of all systems $F=(f_1,\ldots,f_n)$
there exists a proper algebraic hypersurface $D$
such that equality~\eqref{eqBKK} holds for all $F\notin D$.
\end{theorem}

In the special case where $\Delta_1=\cdots=\Delta_n=\Delta$,
the theorem was first established by A.~Kushnirenko,
and the mixed volume reduces to the ordinary volume:
\[
\vol(\Delta_1,\ldots,\Delta_n)=\vol(\Delta).
\]
\subsection{A multidimensional Kac theorem for Laurent polynomials}\label{Laur3}
We now formulate three results that may be viewed as multivariate analogue
of the Kac-type Theorem~\ref{thmL2} for Laurent polynomials.
We begin by introducing the multidimensional counterparts of the
one-dimensional notions discussed in \S\ref{Laur1}.
\begin{enumerate}
\item
A Laurent polynomial in $n$ variables is called \emph{real}
if it takes real values on the compact subtorus
\[
S^n=\{z\in(\C\setminus0)^n:\ z=(\E^{i\theta_1},\ldots,\E^{i\theta_n})\}.
\]

\item
A zero of a Laurent polynomial lying in $S^n$ is called a \emph{real zero}.
More generally, a solution of a system of Laurent polynomials
contained in $S^n$ is called a \emph{real root} of the system.

\item
A Laurent polynomial $P$ is real if and only if
$a_k=\overline{a_{-k}}$ for all $k\in\Z^n$.

\item
The zero set of a real Laurent polynomial is invariant under the involution
$z\mapsto \overline{z}^{-1}$, whose fixed point set is $S^n$.

\item
Let ${\rm Trig}(\Lambda)$ denote the space of trigonometric polynomials of the form
\[
\sum_{k\in\Lambda}
\bigl(\alpha_k\cos(k,\theta)
      +\beta_k\sin(k,\theta)\bigr),
\qquad \alpha_k,\beta_k\in\R.
\]
Restriction to $S^n$ defines a linear isomorphism between the space of real
Laurent polynomials with spectrum $\Lambda$ and the space ${\rm Trig}(\Lambda)$.

\item
Since every vector in $2\pi\Z^n$ is a period of any trigonometric polynomial,
the space ${\rm Trig}(\Lambda)$ may be regarded as a finite-dimensional subspace of
$L^2(S^n)$, equipped with the inner product
\[
(f,g)=\int_{S^n} f(\theta)g(\theta)\,d\theta.
\]

\item
Random vectors in ${\rm Trig}(\Lambda)$, and hence random real Laurent polynomials
with spectrum $\Lambda$, are assumed to be Gaussian with respect to the
standard Gaussian measure $G_\Lambda$ induced by this inner product.

\item
Given real Laurent polynomials $f_1,\ldots,f_n$ with spectra
$\Lambda_1,\ldots,\Lambda_n$, let $N(f_1,\ldots,f_n)$ denote the number of
isolated real roots of the system $f_1=\cdots=f_n=0$.
The \emph{expected number of real roots} of the random system
$f_1=\cdots=f_n=0$ is defined by
\begin{equation}\label{eqExpectation}
\int_{{\rm Trig}(\Lambda_1)\times\cdots\times{\rm Trig}(\Lambda_n)}
N(f_1,\ldots,f_n)\,
dG_{\Lambda_1}\cdots dG_{\Lambda_n}.
\end{equation}
\end{enumerate}
\par\smallskip
It turns out that the expected number of real roots of a random system of
$n$ real Laurent polynomials in $n$ variables with prescribed spectra
can be expressed in terms of mixed volumes of certain $n$-dimensional
ellipsoids, called \emph{Newton ellipsoids}; see Theorem~\ref{thmMy1} below.

\begin{definition}\label{dfNEll}
The \emph{Newton ellipsoid} ${\rm Ell}(\Lambda)$ associated with a finite set
$\Lambda\subset\Z^n$ is the ellipsoid in $\R^n$ whose support function is given by
\[
h_\Lambda(x)
=\sqrt{\frac{1}{\#\Lambda}\sum_{\lambda\in\Lambda}\lambda(x)^2},
\qquad x\in\R^n,
\]
where each $\lambda\in\Lambda$ is regarded as a linear functional on $\R^n$.
Recall that the support function of a compact convex set $A\subset(\R^n)^\ast$
is defined by
\[
h_A(x)=\max_{\xi\in A}\xi(x).
\]
\end{definition}
\begin{theorem}\label{thmMy1}
Let $\Lambda_1,\ldots,\Lambda_n$ be centrally symmetric finite subsets of $\Z^n$.
Then the expected number of real roots of a random system of real Laurent
polynomials with spectra $\Lambda_1,\ldots,\Lambda_n$ equals
\[
n!\,\vol\bigl({\rm Ell}(\Lambda_1),\ldots,{\rm Ell}(\Lambda_n)\bigr),
\]
where $\vol(A_1,\ldots,A_n)$ denotes the mixed volume of the convex bodies
$A_1,\ldots,A_n$.
\end{theorem}
\begin{remark}
In the one-dimensional case $n=1$, the Newton ellipsoid ${\rm Ell}(\Lambda)$
reduces to the interval $[-h_\Lambda(1),h_\Lambda(1)]$,
and Theorem~\ref{thmL2}(a) is recovered.
\end{remark}
\begin{remark}
If $\Lambda$ is centrally symmetric, then
\[
{\rm Ell}(\Lambda)\subset {\rm conv}(\Lambda).
\]
Indeed, for 
compact convex sets $A,B$,
the inclusion $A\subseteq B$ is equivalent to the inequality
$h_A\le h_B$ of their support functions.
For $\xi\in\R^n$, the support function of ${\rm conv}(\Lambda)$ is
\[
h_{{\rm conv}(\Lambda)}(\xi)=\max_{\lambda\in\Lambda}\lambda(\xi)
=\max_{\lambda\in\Lambda}|\lambda(\xi)|,
\]
where the second equality uses the symmetry of $\Lambda$.
By Definition~\ref{dfNEll}, $h_\Lambda(\xi)$ is the root mean square of the
numbers $\{|\lambda(\xi)|:\lambda\in\Lambda\}$, which does not exceed their
maximum. This implies the claim.
\end{remark}
\begin{remark}
Mixed volumes of ellipsoids in connection with the distribution of zeros of system
of random equations were first studied in~\cite{ZK}.
\end{remark}
Applying the Bernstein--Kushnirenko--Khovanskii theorem yields the following
consequence; for details see \S\ref{Laur42}.
\begin{theorem}\label{thmMy2}
Let $\mathcal{P}(\Lambda_1,\ldots,\Lambda_n)$ be the probability that a root
of a random system of $n$ real Laurent polynomials with spectra
$\Lambda_1,\ldots,\Lambda_n$ is real. Then
\begin{equation}\label{eqTrig}
\mathcal{P}(\Lambda_1,\ldots,\Lambda_n)
=
\frac{
\text{\vol}\bigl({\rm Ell}(\Lambda_1),\ldots,{\rm Ell}(\Lambda_n)\bigr)
}{
\text{\vol}\bigl({\rm conv}(\Lambda_1),\ldots,{\rm conv}(\Lambda_n)\bigr)
},
\end{equation}
where ${\rm conv}(\Lambda_i)$ is the Newton polyhedron of the $i$-th polynomial.
\end{theorem}

Building upon the asymptotic analysis in \cite{K22}, we consider the limiting behavior of $\mathcal{P}(\Lambda)$ for expanding spectra, which provides a multidimensional generalization of Corollary \ref{corL1}(c).

\begin{theorem}\label{thmMy3}
Let $B_m \subset \mathbb{R}^n$ denote the ball of radius $m$ centered at the origin, and let $\Lambda_m = B_m \cap \mathbb{Z}^n$. Then
\begin{equation}
\lim_{m\to\infty}\mathcal{P}(\Lambda_m) = \left(\frac{\sigma_{n-1}}{\sigma_n}\,\beta_n\right)^{\!n/2},
\end{equation}
where $\sigma_k$ is the volume of the $k$-dimensional unit ball and
\[
\beta_n = \int_{-1}^1 x^2 (1-x^2)^{(n-1)/2} \, dx.
\]
\end{theorem}

Numerical values of the constant $\beta_n$ for small $n$ are provided in the table below.

%\begin{theorem}\label{thmMy2}
%Let $\mathcal P(\Lambda_1,\ldots,\Lambda_n)$ denote the probability that a root
%of a random system of $n$ real Laurent polynomials with spectra
%$\Lambda_1,\ldots,\Lambda_n$ is real.
%Then
%\begin{equation}\label{eqTrig}
%\mathcal P(\Lambda_1,\ldots,\Lambda_n)
%=
%\frac{
%\vol\bigl({\rm Ell}(\Lambda_1),\ldots,{\rm Ell}(\Lambda_n)\bigr)
%}{
%\vol\bigl({\rm conv}(\Lambda_1),\ldots,{\rm conv}(\Lambda_n)\bigr)
%},
%\end{equation}
%where ${\rm conv}(\Lambda_i)$ denotes the Newton polyhedron of the $i$-th polynomial.
%\end{theorem}
%%
%In~\cite{K22}, the limiting value of the probability $\mathcal P(\Lambda)$ is computed
%for an increasing sequence of spectra~$\Lambda$, yielding
%a multidimensional analogue of Corollary~\ref{corL1}(c).
%%
%\begin{theorem}\label{thmMy3}
%Let $B_m$ be the ball of radius $m$ in $\R^n$ centered at the origin,
%and let $\Lambda_m=B_m\cap\Z^n$.
%Then
%\[
%\lim_{m\to\infty}\mathcal P(\Lambda_m)
%=\left(\frac{\sigma_{n-1}}{\sigma_n}\,\beta_n\right)^{\!n/2},
%\]
%where $\beta_n=\int_{-1}^1x^2(1-x^2)^{(n-1)/2}\,dx$,
%and $\sigma_k$ denotes the volume of the $k$-dimensional unit ball.
%\end{theorem}
%%
%\noindent
%The table below lists the values of $\beta_n$ for $1\le n\le20$.

\begin{table}[h]
  \begin{center}
   \begin{tabular}{r| l l l l l l l l l l}
       {\scriptsize $n$}&{\scriptsize $1$}&{\scriptsize $2$}&{\scriptsize $3$}&{\scriptsize $4$}&{\scriptsize $5$}&{\scriptsize $6$}&{\scriptsize $7$}&{\scriptsize $8$}&
       {\scriptsize $9$}&{\scriptsize $10$}\\
\hline
{\scriptsize $\beta_n$}&
      {\scriptsize $\frac{2}{3}$}&
      {\scriptsize $\frac{\pi}{8}$}&
      {\scriptsize $\frac{4}{15}$}&
      {\scriptsize  $\frac{\pi}{16}$}&
      {\scriptsize $\frac{16}{105}$}&
      {\scriptsize $\frac{5\pi}{128}$}&
      {\scriptsize $\frac{32}{315}$}&
      {\scriptsize $\frac{7\pi}{256}$}&
      {\scriptsize $\frac{256}{3465}$}&
      {\scriptsize $\frac{21\pi}{1024}$}
      \\[2pt]
\hline
       {\scriptsize $n$}
       &{\scriptsize $11$}&
       {\scriptsize $12$}&
       {\scriptsize $13$}&
       {\scriptsize $14$}&
       {\scriptsize $15$}&
       {\scriptsize $16$}&
       {\scriptsize $17$}&
       {\scriptsize $18$}&
       {\scriptsize $19$}&
       {\scriptsize $20$}
                     \\[2pt]
{\scriptsize $\beta_n$}&
      {\scriptsize $\frac{512}{9009}$}&
      {\scriptsize $\frac{33\pi}{2048}$}&
      {\scriptsize $\frac{4096}{109395}$}&
      {\scriptsize $\frac{429\pi}{32768}$}&
      {\scriptsize $\frac{2048}{45045}$}&
      {\scriptsize $\frac{715\pi}{65536}$}&
      {\scriptsize $\frac{65536}{2078505}$}&
      {\scriptsize  $\frac{2431\pi}{262144}$}&
      {\scriptsize $\frac{131072}{4849845}$}&
      {\scriptsize $\frac{4199\pi}{524288}$}
     \end{tabular}
  \end{center}
\end{table}
\begin{remark}
The integrand $x^2(1-x^2)^{(n-1)/2}\,dx$ belongs to the class of \emph{Chebyshev differential binomials}.
We recall a classical result of Chebyshev \cite{Ch},
extending earlier observations of Euler,
which states that the binomial differential $x^m(a+bx^n)^p dx$ admits an antiderivative in elementary functions
if and only if at least one of the quantities $p$, $\frac{m+1}{n}$, or $\frac{m+1}{n} + p$ is an integer.
In the present case, these conditions are satisfied for all $n \in \mathbb{N}$: specifically,
for odd $n$ we recover the first case, while for even $n$ the third case applies.
\end{remark}
A further characterization of the asymptotic behavior of $\mathcal{P}(\Lambda_1,\ldots,\Lambda_n)$
for general sequences of expanding spectra is established in Theorem \ref{thmMy4} within \textsection\ref{Laur43}.
%
%
%\begin{remark}
%The differential form $x^2(1-x^2)^{(n-1)/2}\,dx$
%is a \emph{Chebyshev differential binomial}.
%Chebyshev proved in~\cite{Ch} the nonintegrability of the binomial $x^m(a+bx^n)^p dx$ in elementary functions
%except in three special cases identified by L.~Euler.
%For odd~$n$, the above binomial corresponds to the first case,
%and for even~$n$, to the third one.
%\end{remark}
%%
%Below in \textsection\ref{Laur43} we present one more statement about the asymptotics of the probability
%$\mathcal P(\Lambda_1,\ldots,\Lambda_n)$ with increasing spectra $\Lambda_i$; see Theorem \ref{thmMy4}.
%
%
%
\subsection{Proof Theorems \ref{thmMy1}, \ref{thmMy2}, \ref{thmMy3}, \ref{thmMy4}}\label{Laur4}
\subsubsection{Theorem \ref{thmMy1}}\label{Laur41}
The computation of the expected number of roots relies on the integral geometric formula \eqref{eqAK1} established in \cite{AK}
(cf. \cite{KaB, KaCr} for related applications).

Let $V$ be a finite-dimensional subspace of $C^\infty(X)$, where $X$ is a compact differentiable manifold. We assume throughout that the evaluation functionals are non-vanishing; that is,
\begin{equation}\label{eqne01}
    \text{for each } x \in X, \text{ there exists } f \in V \text{ such that } f(x) \neq 0.
\end{equation}
Let $(\cdot, \cdot)$ be a scalar product on $V$, and denote by $G$ the associated Gaussian measure.

\begin{definition}\label{dfthetaTheta}
The \emph{evaluation mapping} $\Theta \colon X \to V^*$ is defined by $\Theta(x)(f) = f(x)$ for $f \in V$. Condition \eqref{eqne01} ensures that $\Theta$ takes values in $V^* \setminus \{0\}$. We further define the normalized mapping $\bar{\Theta} \colon X \to V^*$ by
\begin{equation}
    \bar{\Theta}(x) = \frac{\Theta(x)}{\|\Theta(x)\|_*},
\end{equation}
where $\|\cdot\|_*$ denotes the norm induced by the dual scalar product on $V^*$. Let $D_x \colon T_x X \to V^*$ be the differential of $\bar{\Theta}$ at $x$,
and let $D_x^* \colon V \to T_x^* X$ be its adjoint operator.
\end{definition}

\begin{definition}\label{dfEllV}
Let $B \subset V$ be the unit ball.
The $V$-\emph{ellipsoid} at $x \in X$, denoted by ${\rm Ell}_V(x)$,
is the image $D_x^*(B)$ in the cotangent space $T_x^* X$.
The set ${\rm Ell}_V(x)$ is a centrally symmetric ellipsoid, possibly of degenerate dimension.
\end{definition}

 %Let $V$ be a finite-dimensional space of smooth functions on a compact differentiable manifold $X$.
% Further, we assume that
%\begin{equation}\label{eqne01}
%  \forall  x\in X\, \exists f\in V\colon f(x)\ne0.
%\end{equation}
%Let $(*,*)$ be the scalar product in $V$,
%and let $G$ denote the Gaussian measure in $V$ corresponding to this scalar product.
%%
%\begin{definition}\label{dfthetaTheta}
%Let $V^*$ denote the space of linear functionals on $V$.
%
%(1)
%Let us define the mapping
%$\Theta\colon X\to V^*$,
%such that $\forall f\in V\colon\,\Theta(x)(f)=f(x)$.
%From (\ref{eqne01}) it follows that
%$\forall x\in X\colon\Theta(x)\ne0$.
%
%(2)
%Let $\Theta/|\Theta|\colon X\to V^*$ be a mapping
%defined by
%$$
%x\mapsto\frac{\Theta(x)}{\sqrt{(\Theta(x),\Theta(x))^*}}
%$$
%where $(*,*)^*$ is the scalar product in $V^*$,
%adjoint to the product $(*,*)$ in %the space
%$V$.
%
%(3)
%Let us denote by $D_{x}\colon T_xX\to V^*$ the differential of $\Theta/|\Theta|$ at the point $x$,
%and let $D^*_{x}\colon V\to T_x^*X$ be the linear operator adjoint to $D_{x}$.
%\end{definition}
%%
%\begin{definition}\label{dfEllV}
%Let $B$ be a ball of radius $1$ in $V$ centered at the origin,
%and let ${\rm Ell}_V(x)=D^*_{x}(B)$.
%Let us call ${\rm Ell}_V(x)$ a $V$-\emph{ellipsoid} at the point $x$.
%The set ${\rm Ell}_V(x)$ is an ellipsoid centered at the origin,
%possibly of incomplete dimension,
%in the cotangent space $T_x^*X$.
%\end{definition}
%
We now establish the connection between the expected number of roots and the geometry of the ellipsoids ${\rm Ell}_{V_i}$.
Let $X$ be an $n$-dimensional manifold and let $V_1, \ldots, V_n$ be subspaces of $C^\infty(X)$ satisfying condition \eqref{eqne01}. We extend the notation of the previous section to each $V_i$, denoting the corresponding Gaussian measures by $G_i$.

\begin{definition}\label{dfMGoth}
For a system of random equations $f_1 = \cdots = f_n = 0$ with $f_i \in V_i$, let $N(f_1, \ldots, f_n)$ denote the number of isolated roots.
The expected number of roots is given by
\begin{equation}
    \mathfrak{M}(V_1, \ldots, V_n) = \int_{V_1 \times \cdots \times V_n} N(f_1, \ldots, f_n) \, dG_1 \cdots dG_n.
\end{equation}
\end{definition}

%Let us now proceed to formulating the required integral geometric statement;
%see Proposition \ref{prCrnAK} below.
%Let $\dim X=n$ and $V_1,\ldots,V_n$ be spaces
%with properties similar to those described above for the space $V$.
%Furthermore, we retain all the notations introduced above,
%assigning them an index $i$ equal to the number of the space $V_i$.
%For example, we denote the scalar product and the Gaussian measure in $V_i$
%respectively by $(*,*)_i$ and $G_i$.
%%
%\begin{definition}\label{dfMGoth}
%For $0\ne f_1\in V_1,\ldots,0\ne f_n\in V_n$,
%let $N(f_1,\ldots,f_n)$ be the number of isolated roots of the system of equations
%$f_1=\ldots=f_n=0$.
% Considering $f_1=\ldots=f_n=0$ as a system of random equations, we define the expectation of the number of it's roots as
%$$
%  \mathfrak M(V_1,\ldots,V_n)=\int_{(f_1,\ldots,f_n)\in V_1\times\ldots\times V_n}N(f_1,\ldots,f_n)\:dG_1\ldots dG_n.
%$$
%\end{definition}
%
%
\begin{proposition}\label{prCrnAK} (see  {\rm\cite{AK}})
Let $g$ be the Riemannian metric on $X$,
and let $dg$ be the corresponding Riemannian volume density on $X$.
Then
\begin{equation}\label{eqAK1}
\mathfrak M(V_1,\ldots,V_n)= {n!\over(2\pi)^n} \int_X\vol({\rm Ell}_{V_1}(x),\ldots,{\rm Ell}_{V_n}(x))\: dg(x),
\end{equation}
where the mixed volume of ellipsoids ${\rm Ell}_{V_i}(x)$ in the cotangent space $T_x^*X$
is measured using the scalar product
adjoint to the scalar product in $T_xX$,
corresponding to the Riemannian metric $g$.
\end{proposition}
Assume henceforth that
$X$ is a homogeneous space of a compact Lie group $K$,
 and that the metric $g$ and the subspaces $V_i$ are $K$-invariant.
 Under these symmetries, the set of
 ellipsoids $\{{\rm Ell}_{V_i}(x)\}_{x \in X}$ is $K$-invariant for each $i$.
\begin{corollary}\label{corMain1}
$$
\mathfrak M(V_1,\ldots,V_n)= {n!\over(2\pi)^n} \vol\left({\rm Ell}_{V_1}(x),\ldots,{\rm Ell}_{V_n}(x)\right)\: \vol(X),
$$
for any $x \in X$.
\end{corollary}
%
%
%---------------------
%
%
%
%\begin{proposition}\label{prEll=2}
%The square of the support function of ${\rm Ell}_V(\mathbf{1})$ is given by
%\begin{equation}
%h^2_{{\rm Ell}_V(\mathbf{1})}(\xi) = \frac{1}{\#\Lambda} \sum_{\lambda \in \Lambda} \langle \lambda, \xi \rangle^2.
%\end{equation}
%\end{proposition}
%
%\begin{proof}
%Let $\{f_\lambda\}_{\lambda \in \Lambda}$ be an orthonormal basis for $\text{Trig}(\Lambda)$ consisting of appropriately scaled trigonometric polynomials. By a direct calculation similar to that in Theorem \ref{thmL2}, we have $|\Theta(s)|^2 = \#\Lambda$ for all $s \in S^n$. For any $\xi \in T_s S^n$, the differential satisfies
%\begin{equation}\label{eq2Th}
%|d\Theta(\xi)|^2 = \sum_{\lambda \in \Lambda} \langle \lambda, \xi \rangle^2.
%\end{equation}
%The result for the normalized mapping $\bar{\Theta} = \Theta / |\Theta|$ follows by observing that at the identity, the pullback of the norm under $d\bar{\Theta}$ scales as $(1/\#\Lambda) \sum (\lambda, \xi)^2$. This completes the proof of Theorem \ref{thmMy1}(1). Assertions (2) and (3) follow immediately.
%\end{proof}
%
Now consider the torus $S^n$ as the homogeneous space of its self-action,
let
$V={\rm Trig}(\Lambda)$,
and for $f,g\in V$,
$$(f,g)=\int_{S^n}f(s_1,\ldots,s_n) g(s_1,\ldots,s_n)\:ds_1\ldots ds_n$$
Since $\text{vol}(S^n) = (2\pi)^n$,
Theorem \ref{thmMy1} reduces to the following geometric identification:
\begin{proposition}\label{prEll=}
The Newton ellipsoid ${\rm Ell}(\Lambda)$ coincides with the $V$-ellipsoid ${\rm Ell}_V(\mathbf{1})$
at the identity $\mathbf{1} \in S^n$.
\end{proposition}
The equality of convex bodies is equivalent to the equality of their support functions.
To prove the required equality, the following well-known property of support functions is used.
\begin{lemma}\label{lmh_V}
Let $h_A\colon P^*\to\R$ be a support function
of a convex body $A$ in the space $P$,
and let $F(A)$ be the image of $A$ under the linear operator $F\colon P\to Q$.
Then the support function %$h_{F(A)}\colon Q^*\to\R$
of a convex body $F(A)$
is a pullback of $h_A$ under the adjoint operator $F^*\colon Q^*\to P^*$.
\end{lemma}
From Lemma \ref{lmh_V} it follows that the support function of ellipsoid  ${\rm Ell}_V(\mathbf 1)$
equals the pullback of the function $|v|$
with respect to the linear operator $D_\mathbf 1\colon T_\mathbf 1S^n\to{\rm Trig}(\Lambda)$.
Therefore, Proposition \ref{prEll=} reduces Theorem \ref{thmMy1} to the next statement.
\begin{proposition}\label{prEll=2}
The pullback of the function $|v|^2$
with respect to the linear operator $D_\mathbf 1\colon T_\mathbf 1S^n\to{\rm Trig}(\Lambda)$
is equal to
$
 \frac{1}{\#\Lambda}\sum\nolimits_{\lambda\in\Lambda}\lambda^2(\xi),
$
where $\mathbf 1$ is the identity element of the torus $S^n$.
I.e. the square of the support function of the Newton ellipsoid ${\rm Ell}(\Lambda)$
is the pullback of the function $|v|^2$.
\end{proposition}
%
%\begin{proof}
Consider $\Z^n$ as an integer lattice in $\R^n$.
Let $l$ be a linear functional in $\R^n$,
not equal to $0$ at nonzero points of the set $\Lambda$.
Denote by $\Lambda_+$ the intersection of $\Lambda$ with the
half-space $l>0$.
Next we choose the special orthonormal basis $\{f_\lambda\colon\:\lambda\in\Lambda\}$ in ${\rm Trig}(\Lambda)$:
\[
f_\lambda(\theta)=
\begin{cases}
\dfrac{\sqrt 2}{(2\pi)^{n/2}}\cos(\lambda,\theta), & \lambda\in\Lambda_+,\\[0.6em]
\dfrac{\sqrt 2}{(2\pi)^{n/2}}\sin(\lambda,\theta), & 0\ne\lambda\notin\Lambda_+,\\[0.6em]
1/(2\pi)^{n/2}, & \lambda=0\in\Lambda.
\end{cases}
\]
Just as in the proof of Theorem \ref{thmL2},
using Lemma \ref{lmLin},
we obtain that
\begin{equation}\label{eq1Th}
\forall s\in S^n\colon\:|\Theta(s)|^2=\#\Lambda,
\end{equation}
and, for any tangent vector $\xi\in T_sS^n$,
\begin{equation}\label{eq2Th}
|d\Theta(\xi)|^2=\sum\nolimits_{\lambda\in\Lambda} (\lambda,\xi)^2
%d\Theta(\xi)=\frac{\sqrt 2}{(2\pi)^{n/2}}\sum\nolimits_{\lambda\in\Lambda_+}(-(\lambda,\xi)\sin(\lambda, s) f_{\lambda}+(\lambda,\xi)\cos(\lambda,s) f_{-\lambda})
\end{equation}
From (\ref{eq2Th}) it follows that
that the value of the function $d\Theta^*(|v^2|)\colon T_\mathbf 1S^n\to\R$ on the tangent vector $\xi$ is equal to $\sum\nolimits_{\lambda\in\Lambda} (\lambda,\xi)^2$.
Accordingly, the value of $d(\Theta/|\Theta|)^*(|v|)$ on $\xi$ is equal to
$$
\sqrt{\frac{1}{\#\Lambda}\sum\nolimits_{\lambda\in\Lambda} (\lambda,\xi)^2}
$$
Proposition \ref{prEll=} and Theorem \ref{thmMy1} (1) are proved.
Statements (2) and (3) of Theorem \ref{thmMy1} are direct consequences of statement (1).
\subsubsection{Theorem  \ref{thmMy2}}\label{Laur42}
Let $\P_{i,\C}$ be the projectivization of the space of Laurent polynomials with spectrum $\Lambda_i$.
A system of Laurent polynomials $f_1,\ldots,f_n$ will be identified with a point of the product
$\P_{1,\C}\times\cdots\times\P_{n,\C}$.
Likewise, we identify ${\rm Trig}(\Lambda_i)$ with the real subspace of Laurent polynomials having spectrum $\Lambda_i$,
and hence identify systems of $n$ real Laurent polynomials with points of the space
$\P_1\times\cdots\times\P_n$,
where $\P_i$ denotes the projectivization of ${\rm Trig}(\Lambda_i)$.

Consider the natural embedding
\[
\iota\colon \P_1\times\cdots\times\P_n \longrightarrow \P_{1,\C}\times\cdots\times\P_{n,\C}.
\]
Its image is Zariski dense in $\P_{1,\C}\times\cdots\times\P_{n,\C}$.

By Theorem~\ref{thmBKK}, there exists an algebraic hypersurface
$D\subset \P_{1,\C}\times\cdots\times\P_{n,\C}$
such that for any system $(f_1,\ldots,f_n)\notin D$,
the number of common zeros equals
\[
n!\vol\bigl({\rm conv}(\Lambda_1),\ldots,{\rm conv}(\Lambda_n)\bigr).
\]
Since the image of\/ $\iota$ is Zariski dense,
its inverse image $\iota^{-1}(D)$ is contained in a closed real hypersurface in
$\P_1\times\cdots\times\P_n$.
In particular, $\iota^{-1}(D)$ has measure zero.
Applying Theorem~\ref{thmMy2}, we obtain the desired equality.
\subsubsection{Theorems  \ref{thmMy3}, \ref{thmMy4}}\label{Laur43}
Let $\Delta\subset\mathfrak t^*$ be a centrally symmetric compact convex set, and set
$\Lambda=\Delta\cap{\Z^n}^*$.
We further impose the following condition.

\medskip
\noindent
\textbf{(*)}
If\/ $\dim\Delta=k<n$, then $\Delta$ is contained in a $k$-dimensional subspace
$V_\Delta\subset\mathfrak t^*$ generated by vectors of the lattice ${\Z^n}^*$.
\par\smallskip
\medskip
Recall that a centrally symmetric finite set $\Lambda\subset{\Z^n}^*$ determines
the Newton ellipsoid ${\rm Ell}(\Lambda)\subset\mathfrak t^*$ with support function
$h_\Lambda=\sqrt{F_\Lambda}$, where
\[
F_\Lambda(\xi)=\frac{1}{\#\Lambda}\sum_{\lambda\in\Lambda}\lambda^2(\xi).
\]
For $m>0$ define
\[
\Delta_m=m\Delta,\qquad
\Lambda_m=\Delta_m\cap{\Z^n}^*,\qquad
N_{\Lambda,m}=\#\Lambda_m,
\]
and
\[
F_{\Lambda,m}(\xi)=\frac{1}{N_{\Lambda,m}}\sum_{\lambda\in\Lambda_m}\lambda^2(\xi).
\]
The function $h_{\Lambda,m}=\sqrt{F_{\Lambda,m}}$ is the support function of the ellipsoid
${\rm Ell}(\Lambda_m)$.

Condition {\rm(*)} implies that ${\Z^n}^*\cap V_\Delta$ is a full-rank lattice in the
$k$-dimensional space $V_\Delta$.
Fix a volume form $dx$ on $V_\Delta$ normalized so that the unit cube of this lattice has volume $1$.
For $\xi\in\mathfrak t$ set
\[
F_\Delta(\xi)=\frac{1}{\vol_k(\Delta)}\int_{\Delta}\langle x,\xi\rangle^2\,dx.
\]
Then $F_\Delta\colon\mathfrak t\to\R$ is a nonnegative quadratic form.
Let $h_\Delta=\sqrt{F_\Delta}$ and denote by ${\rm Ell}(\Delta)$ the ellipsoid in $\mathfrak t^*$
with support function $h_\Delta$.
\begin{theorem}\label{thmMy4}
Assume that condition {\rm(*)} holds for convex sets $\Delta_i$, and let
$\Lambda_i=\Delta_i\cap{\Z^n}^*$.
Then
\begin{equation}\label{eqMain}
\lim_{\inf(m_1,\ldots,m_n)\to\infty}
{\mathcal P}\bigl((\Lambda_1)_{m_1},\ldots,(\Lambda_n)_{m_n}\bigr)
=
\frac{\vol({\rm Ell}(\Delta_1),\ldots,{\rm Ell}(\Delta_n))}
{\vol(\Delta_1,\ldots,\Delta_n)} .
\end{equation}
\end{theorem}
The proof requires several preliminary results.

\begin{lemma}\label{lmLim}
As $m\to\infty$, the functions $\frac{1}{m^2}F_{\Lambda,m}$ converge locally uniformly
to $F_\Delta$.
\end{lemma}

\begin{proof}
Let $\dim\Delta=k$.
Then
\[
\frac{1}{m^2}F_{\Lambda,m}(\xi)
=\frac{1}{m^2N_{\Lambda,m}}\sum_{\lambda\in\Lambda_m}\langle\xi,\lambda\rangle^2
=\frac{m^k}{N_{\Lambda,m}}
\sum_{\alpha\in\frac{\Lambda_m}{m}}\langle\xi,\alpha\rangle^2\,\frac{1}{m^k}.
\]
Since $N_{\Lambda,m}\asymp m^k\vol_k(\Delta)$, the last sum is a Riemann sum for the integral
$\int_\Delta\langle x,\xi\rangle^2\,dx$ over the partition of $\Delta$ with nodes
$\frac{\Lambda_m}{m}$.
From this we obtain that $\frac{1}{m^2}F_{\Lambda,m}(\xi)\to F_\Delta(\xi)$.
The lemma is proven.
\end{proof}
\begin{corollary}\label{corLim}
As $m\to\infty$:
\begin{enumerate}
\item the functions $\frac{1}{m}h_{\Lambda,m}$ converge locally uniformly to the support
function $h_\Delta$ of ${\rm Ell}(\Delta)$;
\item the ellipsoids $\frac{1}{m}{\rm Ell}(\Lambda_m)$ converge to ${\rm Ell}(\Delta)$
in the Hausdorff topology.
\end{enumerate}
\end{corollary}
\begin{proof}
Both assertions follow immediately from Lemma~\ref{lmLim} and the identity
$h_\Delta=\sqrt{F_\Delta}$.
\end{proof}
\begin{lemma}\label{lmLimConv}
If condition {\rm(*)} holds for a convex body $\Delta$, then
$\frac{1}{m}{\rm conv}(\Lambda_m)$ converges to $\Delta$ in the Hausdorff topology as $m\to\infty$.
\end{lemma}
\begin{proof}
This is an immediate consequence of the definition of $\Lambda_m$.
\end{proof}
\medskip
\noindent
\emph{Proof of Theorem~\ref{thmMy4}.}
Let $M_i=\bigl((m_1)_i,\ldots,(m_n)_i\bigr)$ be a sequence of $n$-tuples of positive integers
such that for each $k$ the sequence $m_{k,1},m_{k,2},\ldots$ tends to infinity.
Apply Theorem~\ref{thmMy2} to the spectra
$(\Lambda_1)_{(m_1)_i},\ldots,(\Lambda_n)_{(m_n)_i}$
and pass to the limit using Corollary~\ref{corLim}(2) and Lemma~\ref{lmLimConv}.
This yields~\eqref{eqMain}.
\qed
\begin{corollary}\label{corHom}
Let $\alpha_1,\ldots,\alpha_n>0$.
Under the substitutions $\Delta_i\mapsto\alpha_i\Delta_i$
(condition {\rm(*)} being preserved),
the asymptotic behaviour of
${\mathcal P}(\Lambda_{m_1},\ldots,\Lambda_{m_n})$ remains unchanged.
\end{corollary}

\begin{proof}
By Lemma~\ref{lmLim}, $h_{\alpha\Delta_i}=\alpha h_{\Delta_i}$.
Hence both the numerator and denominator in~\eqref{eqMain} are multiplied by
$\alpha_1\cdots\alpha_n$.
\end{proof}
\medskip
\noindent
\emph{Proof of Theorem~\ref{thmMy3}.}
Let $B_m\subset\R^n$ be the ball of radius $m$ centred at the origin.
By Theorem~\ref{thmMy4},
\[
\lim_{m\to\infty}{\mathcal P}(B_m\cap\Z^n)
=\frac{\vol({\rm Ell}(B_1))}{\sigma_n},
\]
where $\sigma_k$ denotes the volume of the $k$-dimensional unit ball.
Thus the theorem reduces to the following statement.

\begin{proposition}\label{prLimBeta}
\[
{\rm Ell}(B_1)=\sqrt{\frac{\sigma_{n-1}}{\sigma_n}\beta_n}\,B_1.
\]
\end{proposition}

\begin{proof}
By definition, the support function of ${\rm Ell}(B_1)$ is $h_{B_1}=\sqrt{F_{B_1}}$, where
\[
F_{B_1}(\xi)=\frac{1}{\sigma_n}\int_{B_1}
(x_1\xi_1+\cdots+x_n\xi_n)^2\,dx_1\cdots dx_n.
\]
Since $F_{B_1}(\xi)$ depends only on $|\xi|$, the ellipsoid ${\rm Ell}(B_1)$ is a ball of
radius $\sqrt{F_{B_1}(\xi_0)}$ with $\xi_0=(1,0,\ldots,0)$.
Evaluating the integral yields
\[
F_{B_1}(\xi_0)
=\frac{\sigma_{n-1}}{\sigma_n}
\int_{-1}^1 x_1^2(1-x_1^2)^{\frac{n-1}{2}}\,dx_1
=\frac{\sigma_{n-1}}{\sigma_n}\beta_n,
\]
which proves the claim.
\end{proof}

\medskip
\noindent
\begin{corollary}
\[
\lim_{\inf(m_1,\ldots,m_n)\to\infty}
{\rm real}_n(B_{m_1}\cap\Z^n,\ldots,B_{m_n}\cap\Z^n)
=
\left(\frac{\sigma_{n-1}}{\sigma_n}\beta_n\right)^{\frac{n}{2}}.
\]
\end{corollary}

\begin{proof}
The mixed volume of $n$ balls of radii $r_1,\ldots,r_n$ equals
$r_1\cdots r_n\,\sigma_n$.
The statement now follows from Theorem~\ref{thmMy2} and
Proposition~\ref{prLimBeta}.
\end{proof}
\section{Kac type theorems in the context of groups}\label{group}
\subsection{Introduction}\label{groupIntro}
Let $\pi$ be a real finite-dimensional representation of a compact Lie group $K$
in a real vector space $E$, and let ${\rm Trig}(\pi)$ denote the vector space of
functions on $K$ spanned by the matrix elements of~$\pi$.

Denote by $K^\C$ the complexification of the group $K$.
Recall that $K^\C$ exists, is unique up to isomorphism, and is characterized by
the following properties:
\begin{enumerate}
\item $K^\C$ is a connected complex Lie group with
$\dim_\C K^\C=\dim K$;
\item the Lie algebra of $K^\C$ is the complexification of the Lie algebra of $K$;
\item $K$ is a maximal compact subgroup of $K^\C$.
\end{enumerate}
For example, $(\C\setminus\{0\})^n$ and $GL(n,\C)$ are complexifications of the
torus $S^n$ and of the unitary group $U(n)$, respectively.

The representation $\pi$ extends uniquely to a holomorphic representation of
$K^\C$ in the complexified space $E\otimes_\R\C$.
Consequently, any $\pi$-polynomial $f\in{\rm Trig}(\pi)$ admits a holomorphic
extension $f^\C$ to $K^\C$.
We refer to such extensions as \emph{real $\pi$-polynomials on the group $K^\C$}.
The space of real $\pi$-polynomials is a real vector space of dimension
$\dim{\rm Trig}(\pi)$.
A zero of a $\pi$-polynomial $f$ lying in $K$ is called a \emph{real root} of $f$,
or equivalently, a real root of the corresponding real $\pi$-polynomial $f^\C$.
\par\smallskip
In the case of Laurent polynomials, the notion of a $\pi$-polynomial specializes
as follows.
Let $\pi_0$ be the trivial representation of the torus $S^n$ in $\R$.
For a nonzero vector $m\in\Z^n$, define a real representation $\pi_m$ of $S^n$ by
\[
\pi_m(\theta)=
\begin{pmatrix}
\cos(m,\theta) & \sin(m,\theta)\\
-\sin(m,\theta) & \cos(m,\theta)
\end{pmatrix}.
\]
For each unordered pair $(m,-m)$, set $\pi_{(m,-m)}=\pi_m$.
Since the real representations $\pi_m$ and $\pi_{-m}$ are equivalent, this
notation is well defined.

Let $\Lambda\subset\Z^n$ be a finite centrally symmetric set, and denote by
$\Lambda'$ the set of unordered pairs $(m,-m)$ with $m\in\Lambda$.
Consider the real representation of the torus $S^n$ given by
\begin{equation}\label{eqpi1}
\pi_\Lambda=\bigoplus_{(m,-m)\in\Lambda'}\pi_{(m,-m)}.
\end{equation}
Then (cf.\ the definitions in the beginning of \S\ref{Laur3}) the following hold:
\begin{enumerate}
\item $\pi_\Lambda$-polynomials on $S^n$ are precisely the trigonometric polynomials
\[
f(e^{i\theta})=
\sum_{m\in\Lambda}
\bigl(\alpha_m\cos(m,\theta)+\beta_m\sin(m,\theta)\bigr),
\qquad \alpha_m,\beta_m\in\R;
\]
\item the space of real $\pi_\Lambda$-polynomials coincides with the space of
real Laurent polynomials with spectrum $\Lambda$.
\end{enumerate}
In \S\ref{group3} and \S\ref{group4} we formulate analogues of
Theorems~\ref{thmMy1}, \ref{thmMy2}, and~\ref{thmMy3} for $\pi$-polynomials
associated with representations of a compact Lie group $K$.
Their proofs follow the same general scheme as in \S\ref{Laur3};
accordingly, we provide only a sketch here.
For a complete exposition, see~\cite{K25}.
\subsection{Theorem BKK for reductive groups}\label{group2}
Let $K^\C$ be the complexification of a compact group $K$.
It is a complex, connected, $n$-dimensional reductive Lie group,
such that $K$ is a maximal compact subgroup of $K^\C$.
Let ${\rm Trig}^\C(\mu)$ denote the space of functions on the group $K^\C$,
consisting of linear combinations of the matrix elements of the complex representation $\mu$ of the group $K^\C$.
Let $\mu_1,\ldots,\mu_n$ be finite-dimensional holomorphic representations of $K^\C$.
To any system of $n$ nonzero functions $f_i\in{\rm Trig}^\C(\mu_i)$
we associate the point
$\iota(f_1,\ldots,f_n)=(\C f_1)\times\ldots\times(\C f_n)\in\P_{1,\C}\times\ldots\times\P_{n,\C}$,
where $\P_{i,\C}$ is the complex projective space whose points are
one-dimensional subspaces of ${\rm Trig}^\C(\mu_i)$.
Next, we use the following standard statement from algebraic geometry.
\begin{proposition}\label{prConstAl}
There exist a number $N(\mu_1,\ldots,\mu_n)$ and an algebraic hypersurface $H$ in $\P_{1,\C}\times\ldots\times\P_{n,\C}$ such that the following holds.
For any $n$ functions $f_i\in{\rm Trig}(\Pi_{i,\C})$ with $\iota(f_1,\ldots,f_n)\not\in H$,
the number of their common zeros is $N(\mu_1,\ldots,\mu_n)$.
\end{proposition}
Below
we give a geometric formula for $N(\mu_1,\ldots,\mu_n)$; see Theorem \ref{thmRed}.
This formula differs from its versions in \cite{K87,Br}
by one application of the Weyl integration formula; see \cite{K25}.

Let $T^k$, $\mathfrak t$, $\mathfrak t^*$, $W^*$, and $\mathfrak C^*$
denote a maximal torus in $K$,
its Lie algebra, the space of linear functionals on $\mathfrak t$,
the Weyl group acting in the space $\mathfrak t^*$,
and the  Weyl chamber in $\mathfrak t^*$,
respectively.

Consider the decomposition
$$
\mu=\bigoplus_{\lambda\in\Lambda\subset\Z^k\cap\mathfrak C^*,\:0<m_\lambda\in\Z} m_\lambda\:\mu_\lambda
$$
of a representation $\mu$ into a sum of irreducible representations $\mu_\lambda$
with highest weights $\lambda$ and multiplicities $m_\lambda$.
\begin{definition}\label{dfweightedpolyhedron}
(1) We denote the Weyl group orbit of a point $\lambda\in\mathfrak t^*$ by $W^*(\lambda)$.
The compact convex set
$$
\Delta(\mu)={\rm conv}\left(\bigcup_{\lambda\in\Lambda}W^*(\lambda)\right)
$$
is called the \emph{weight polytope of the representation} $\mu$.

(2)
Let $\mathfrak N(\mu)\subset \mathfrak k^*$ denote the union of coadjoint orbits of $K$
intersecting the weight polytope $\Delta(\mu)$
(we identify $\mathfrak t^*$ with the set of fixed points of the coadjoint action of $T^k$ on $\mathfrak k^*$).
In what follows, we call $\mathfrak N(\mu)$ the \emph{Newton body of the representation} $\mu$.
\end{definition}
\begin{corollary}\label{corNewtTor}
{\rm(1)} Let $\mu_T$ be a restriction of the representation $\mu$ to the maximal torus $T^k$ in $K$,
and let $\Lambda_T\subset \mathfrak t^*$ be the set of weights of the torus representation $\mu_T$.
Then $\Delta(\mu)={\rm conv}(\Lambda_T)$.

{\rm (2)}
The set $\mathfrak N(\mu)$ is a convex compact set in the space $\mathfrak k^*$,
consisting of linear functionals on the Lie algebra $\mathfrak k$ of the group $K$.

{\rm (3)}
$\Delta(\mu)=\pi(\mathfrak N(\mu))$,
where $\pi$ is the projection $\mathfrak k^*\to \mathfrak t^*$.

{\rm (4)}
$\mathfrak N(\mu\otimes \pi)=\mathfrak N(\mu)+ \mathfrak N(\pi)$.
\end{corollary}
\begin{proof}
Statement (1) follows from the theory of highest weights.
It is known that for any $\zeta\in\mathfrak t^*$,
the projection $\pi$ of the coadjoint orbit ${\rm Ad}(K)(\zeta)$ onto $\mathfrak t^*$
coincides with the convex hull of the Weyl group orbit $W^*(\zeta)$; see \cite{K}.
From this, statements (2) and (3) follow.
Statement (4) follows from the standard properties of representation weights.
\end{proof}
\begin{theorem}\label{thmRed}
Let

$\bullet$\
$\tau$ be the metric in $\mathfrak k^*$,
invariant under the coadjoint action of $K$
and such that the area of the basic parallelepiped of the character lattice in $\mathfrak t^*$ is equal to one

$\bullet$\ $\rho=\frac{1}{2}\sum_{\beta\in R^+}\beta$,
where $R^+$ is the set of positive roots corresponding to the Weyl chamber $\mathfrak C^*$

$\bullet$\ $P$ be the polynomial on $\mathfrak k$,
defined as $P(\lambda)=\prod_{\beta\in R^+}(\lambda,\beta)$

Then
$$
N(\mu_1,\ldots,\mu_n)=\frac{n!}{P^2(\rho)}\:\vol_\tau\left(\mathfrak N(\mu_1),\ldots,\mathfrak N(\mu_n)\right),
$$
\end{theorem}
\begin{remark}
If $K=T^n$, then Theorems \ref{thmRed} and \ref{thmBKK} coinside.
\end{remark}
\subsection{Newton ellipsoid of group representation}\label{group3}
For a real representation $\pi$ of a group $K$,
we denote by ${\rm Trig}(\pi)$ the real vector space
consisting of linear combinations of matrix elements of the representation.
Next we regard ${\rm Trig}(\pi)$ as a subspace in $L_2(K)$,
that is, a subspace of the space of functions on $K$ endowed with the scalar product
$(\phi, \psi)=\int\nolimits_K \phi \psi \:dg$,
where $g$ is an invariant Riemannian metric on $K$
such that the Riemannian volume of $K$ is equal to one.
In what follows, we use the notation from \textsection\ref{Laur41}
with $X=K$ and $V_1=\ldots=V_n={\rm Trig}(\pi)$.
\begin{definition}\label{dfEllNGr}
We call the ellipsoid ${\rm Ell}(\pi)={\rm Ell}_{{\rm Trig}(\pi)}(\mathbf 1)$,
where $\mathbf 1$ is the identity element of $K$,
the Newton ellipsoid of the representation of $\pi$.
\end{definition}
Let us list some immediate consequences of Definition \ref{dfEllNGr}.
\begin{corollary}\label{corEllNGr}
{\rm(1)}
The Newton ellipsoid is an ellipsoid in the space $\mathfrak k^*$,
consisting of linear functionals on the Lie algebra $\mathfrak k$ of the group $K$.

{\rm(2)}
The ellipsoid ${\rm Ell}(\pi)$ is invariant under the coadjoint action of $K$.

{\rm(3)}
If the group $K$ is simple,
then the ellipsoid ${\rm Ell}(\pi)$ is a ball in $\mathfrak k^*$ for any coadjointly invariant metric in $\mathfrak k^*$.
(This property of the ellipsoid ${\rm Ell}(\pi)$ is used in calculating the asymptotics of the number of real roots in Theorem \ref{thmAntiKac2}.)
\end{corollary}
\begin{definition}\label{dfEverGr}
Let $\pi_1,\ldots,\pi_n$ be %finite-dimensional
real representations $\pi_1,\ldots,\pi_n$ of the group $K$.
We consider the set of functions
$f_1\in{\rm Trig}(\pi_1),\ldots,f_n\in{\rm Trig}(\pi_n)$
as a set of independent random vectors in the spaces ${\rm Trig}(\pi_1),\ldots,{\rm Trig}(\pi_n)$
with Gaussian measures $G_1,\ldots,G_n$.
We consider $f_1,\ldots,f_n$ as a set of independent random vectors in the spaces ${\rm Trig}(\pi_i)$
with Gaussian measures $G_1,\ldots,G_n$.
The expectation of the number of isolated roots of the random system of equations $f_1=\ldots=f_n=0$
is denoted by $\mathfrak M(\pi_1,\ldots,\pi_n)$.
\end{definition}
\begin{theorem}\label{thmKacGr}
For finite-dimensional real representations $\pi_1,\ldots,\pi_n$,
$$
\mathfrak M(\pi_1,\ldots,\pi_n)=\frac{n!}{(2\pi)^n}\:\vol_\tau({\rm Ell}(\pi_1),\ldots,{\rm Ell}(\pi_n)).
$$
\end{theorem}
\begin{proof}
Comparing Definitions \ref{dfMGoth} and \ref{dfEverGr},
we find that the desired statement follows from Corollary \ref{corMain1} (2).
\end{proof}
\begin{remark}
The factor $\frac{1}{(2\pi)^n}$ in Theorem \ref{thmMy1} is absent.
This is explained by the following difference in formulations.
In Theorem \ref{thmMy1}, the space ${\rm Trig}(\Lambda)$ is considered with the scalar product $\int_{S^n}\phi(\theta) \psi(\theta) d\theta$,
while in Theorem \ref{thmKacGr}, the analogous space ${\rm Trig}(\pi)$
has the scalar product $\frac{1}{(2\pi)^n}\int_{S^n}\phi(\theta) \psi(\theta) d\theta$.
\end{remark}
%
%We define a mapping
%$$
%\forall f\in{\rm Trig}(\pi)\colon\:\left(\Theta(\rho),f\right)=\frac{1}{\sqrt N}f(\rho),
%$$
%where $N=\dim{\rm Trig}(\pi)$.
%It follows from Lemma \ref{lmLin} that
%the set $\Theta(K)$ in ${\rm Trig}(\pi)$
%is contained in a sphere of radius $1$ centered at the origin.
%
%Let $F_\pi$ denote the pullback of the quadratic form $(\psi,\psi)$
%under the mapping
%$d\Theta\colon\mathfrak k\to{\rm Trig}(\pi)$,
%where $d\Theta\colon\mathfrak k\to{\rm Trig}(\pi)$ is the differential of the mapping $\Theta$
%at the unit of the group $K$.
%By definition,
%the quadratic form $F_\pi$ on the Lie algebra $\mathfrak k$ is nonnegative and
%invariant under the adjoint action of the group $K$.
%We set $h_\pi(\xi)=\sqrt{F_\pi}$.
%By construction, for the function $h_\pi\colon\mathfrak k\to{\rm Trig}(\pi)$,
%the following is true:
%
%(1) the function $h_\pi$ is non-negative, convex and invariant under the adjoint action $K$
%
%(2)$\forall (\xi\in\mathfrak k, t\in\R)\colon\: h_\pi(t\xi)=|t|h_\pi(\xi)$.
%%
%\begin{corollary}\label{corNEll}
%The function $h_\pi$ is the support function of some ellipsoid ${\rm Ell}(\pi)$
%in the space $\mathfrak k^*$.
%The ellipsoid ${\rm Ell}(\pi)$
%is invariant with respect to the adjoint action $K$.
%\end{corollary}
%%
%Легко проверить,
%что эллипсоид ${\rm Ell}(\pi)$ в пространстве $\mathfrak k^*$
%не зависит от выбора метрики
%We will call ${\rm Ell}(\pi)$
%\emph{the Newton ellipsoid of the representation} of $\pi$.
%
%
\subsection{Asymptotics of the  proportion of real roots}\label{group4}
Denote by ${\mathcal P}(\pi)$
the probability that a root of a random system of $n$ $\pi$-polynomials is real.
We consider the asymptotics of ${\mathcal P}(\pi)$ for increasing representation $\pi$ of the simple group $K$.

Recall that by $W^*$ and $\mathfrak C^*$ we denote, respectively, the Weyl group and the Weyl chamber in $\mathfrak k^*$.
To formulate the main result, we will need some general concepts about real representations of compact groups.

$\bullet$\
For any $\lambda\in\mathfrak C^*$ there exists an unique $\lambda'\in\mathfrak C^*$
such that $\lambda'\in W^*(-\lambda)$.
For example, if the group $W^*$ contains the central symmetry map, then $\lambda'=\lambda$.
If $K=S^n$,
then $\lambda'=-\lambda$.
The mapping $\mathfrak C^*\to\mathfrak C^*$,
defined as $\lambda\mapsto\lambda'$,
is an involution.
An unordered pair $(\lambda,\lambda')$ is called a \emph{symmetric pair}.
A subset $\Lambda\subset\mathfrak C^*$ is called \emph{symmetric}
if $\lambda\in\Lambda$ implies $\lambda'\in\Lambda$.
For a symmetric set $\Lambda$ denote by $\Lambda'$
the set of symmetric pairs $\{(\lambda,\lambda'):\lambda\in\Lambda\}$.

$\bullet$\
The following statement is an analogue of the highest-weight theory for real representations of the group $K$;
see \cite[Chapter IX, Appendix II]{B}.
\begin{assert}
There is a bijective mapping $(\lambda,\lambda')\mapsto\pi_{\lambda,\lambda'}$
from the set of symmetric pairs $(\lambda,\lambda')$ in $(\mathfrak C^*\cap\Z^k)\times(\mathfrak C^*\cap\Z^k)$
onto the set of irreducible real representations of the group $K$.
It classifies irreducible representations into three types: real, complex, and quaternionic:

(i)\ \emph{real type} $\pi_{\lambda,\lambda'}$: $\:\pi_{\lambda,\lambda'}\otimes_\R\C=\mu_\lambda$ and $\lambda=\lambda'$

(ii)\ \emph{quaternionic type} $\pi_{\lambda,\lambda'}$: $\:\pi_{\lambda,\lambda'}\otimes_\R\C=\mu_\lambda\oplus\mu_\lambda$ and $\lambda=\lambda'$

(iii)\ \emph{complex type} $\pi_{\lambda,\lambda'}$: $\:\pi_{\lambda,\lambda'}\otimes_\R\C=\mu_\lambda\oplus\mu_{\lambda'}$ and $\lambda\ne\lambda'$,

\noindent
where $\mu_\lambda$ denotes the complex irreducible representation of $K$ with highest weight $\lambda$.
\end{assert}

Let $B$ be a ball in $\mathfrak t^*$ of radius one centered at the origin, and
let $\Lambda(B)=B\cap\mathfrak C^*\cap\Z^k$.
Consider a sequence of sets $\Lambda(mB)$ and the corresponding sequence of representations
\begin{equation}\label{eqpim}
  \pi_m=\bigoplus_{(\lambda,\lambda')\in\Lambda'(mB)}\pi_{\lambda,\lambda'},
\end{equation}
We consider the holomorphic extensions of $\pi_m$-polynomials
to polynomials in ${\rm Trig}(K^\C)$ as analogues of real Laurent polynomials of degree $m$.
\begin{theorem}\label{thmAntiKac2}
Let $K$ be a simple group,
and let $P$, $\rho$ be taken from the formulation of Theorem \ref{thmRed}.
Then
$$
\lim_{m\to\infty} {\mathcal P}(\pi_m)=
\frac{P^2(\rho)}{(2\pi)^n(n+2)^{n/2}(\alpha,\alpha+2\rho)^{n/2}},
$$
where $\alpha$ is the highest root of $K$ (i.e., the highest weight of the adjoint representation $\mu_\alpha$).
\end{theorem}
\section{FTA for exponential sums}\label{exp}
\subsection{Exponential sums of one variable}\label{exp1}
An \emph{exponential sum} is a complex-valued function of the form
\[
  f(z)=\sum_{\lambda\in\Lambda} c_\lambda\,\E^{\bar\lambda z},
\]
where $\Lambda\subset\C$ is a finite set and $c_\lambda\in\C$.
The set $\Lambda$ is called the \emph{spectrum} of $f$, and its convex hull
\[
  \Delta=\conv(\Lambda)
\]
is referred to as the \emph{Newton polygon} of~$f$.
If $\#\Lambda>1$, the zero set of an exponential sum is infinite; for example,
the function $f(z)=\E^{2\pi i z}-1$ vanishes precisely on~$\Z$.

The following result may be viewed as an analogue of the Fundamental Theorem
of Algebra for exponential sums.

\begin{theorem}\label{thmexp}
Let $N(f,r)$ denote the number of zeros of an exponential sum~$f$ in the disk
$B_r=\{z\in\C:\,|z|\le r\}$.
Then
\begin{equation}\label{expMain}
  N(f,r)=\frac{r}{2\pi}\,l(f)+O(1),
\end{equation}
where $l(f)$ denotes the semiperimeter of the Newton polygon $\Delta$ of~$f$.
\end{theorem}

\begin{remark}
If the polygon $\Delta$ degenerates to a line segment, its perimeter converges
to twice the length of that segment.
In this case, the quantity $l(f)$ in~\eqref{expMain} coincides with the length
of~$\Delta$.
\end{remark}

\begin{remark}
Under the substitution $z\mapsto\E^z$, a Laurent polynomial is transformed into
an exponential sum.
Thus Theorem~\ref{thmexp} recovers the classical Fundamental Theorem of Algebra.
\end{remark}
We begin the proof of Theorem~\ref{thmexp} with an auxiliary construction.

Let $K\subset\C$ be a smooth compact curve.
Assume first that $K$ does not intersect the zero set of an exponential sum~$g$.
Let ${\rm Arg}_{g,K}$ denote a branch of the multivalued function $\arg g(z)$
defined along~$K$.
The differential $d{\rm Arg}_{g,K}$ is independent of the chosen branch.
We set
\[
  {\rm Arg}_g(K)=\int_K d{\rm Arg}_{g,K}.
\]
For an arbitrary curve $K$, define
\[
  {\rm Arg}_g(K)=\overline{\lim}_{z\to0}{\rm Arg}_g(z+K).
\]
Clearly, $|{\rm Arg}_g(K)|<\infty$.
\begin{lemma}\label{lmFTAexp1}
Let $A$ and $K$ be a compact set and a smooth compact curve in $\C$, respectively.
Then, for any fixed exponential sum $f$, there exists $C>0$ such that

{\rm(1)}
$\forall z\in\C\colon\,|{\rm Arg}_f(z+K)|\le C$

{\rm(2)} For the number of zeros $N(f,z+A)$ of the exponential sum $f$ in the set $z+A$ it is true that
$\forall z\in\C\colon\,N(f,z+A)\le C$.
\end{lemma}

\begin{proof}
Let $\Lambda$ be the spectrum of $f$, and let $E$ denote the vector space of
exponential sums with spectrum~$\Lambda$.
Let $\mathbb E=\mathbb P(E)$ be its projectivization; thus points of
$\mathbb E$ correspond to exponential sums defined up to multiplication by a
nonzero scalar.

Translations by $w\in\C$ act naturally on $E$ via
\[
  (w\cdot g)(z)=g(z+w),
\]
and hence induce a projective representation
\[
  \rho\colon \C \longrightarrow \mathrm{Aut}(\mathbb E).
\]

Assume, towards a contradiction, that no such constant $C$ exists.
Then there is a sequence $\{z_k\}\subset\C$ such that at least one of the
quantities
\[
  |{\rm Arg}_f(z_k+K)|, \qquad N(f,z_k+A)
\]
tends to infinity as $k\to\infty$.

Passing to a subsequence if necessary, we may assume that
$\rho(z_k)(f)$ converges in $\mathbb E$ to a limit $f_\infty$.
It would then follow that either
$|{\rm Arg}_{f_\infty}(K)|=\infty$ or $N(f_\infty,A)=\infty$,
which is impossible.
This contradiction completes the proof.
\end{proof}

\begin{proof}[Proof of Theorem~\ref{thmexp}]
Let $\Delta_1,\dots,\Delta_N$ be the sides of the Newton polygon
$\Delta=\conv(\Lambda)$, and let $r_1,\dots,r_N$ be the rays generated by the
outward normal vectors to these sides.
For each $j$, denote by $U_{j,R}$ the $R$-neighborhood of $r_j$ in $\C$.

For $R$ sufficiently large, all zeros of $f$ are contained in
$\bigcup_{j=1}^N U_{j,R}$.
It therefore suffices to prove that, for each $j$,
\begin{equation}\label{side}
  N(f,U_{j,R}\cap B_r)
  =\frac{r}{2\pi}\,{\rm length}(\Delta_j)+O(1).
\end{equation}
Let $D_j$ denote the boundary of the domain
$U_{j,R}\cap B_r$.
By the argument principle,
\[
  N(f,U_{j,R}\cap B_r)
  =\frac{1}{2\pi}\,{\rm Arg}_f(D_j).
\]
Choose coordinates $(x,y)$ in $\C$ so that the ray $r_j$ coincides with the
positive $y$-axis.
Then $D_j$ is the union of four curves:
the vertical segments $V_-(r)$ and $V_+(r)$ with endpoints
$(-R,0),(-R,r)$ and $(R,0),(R,r)$, respectively,
an arc $c_l$ of the circle of radius $R$ centered at the origin lying below the
$x$-axis,
and an arc $c_u(r)$ of the circle of radius $r$ centered at the origin lying
above the segment with endpoints
$(-R,\sqrt{r^2-R^2})$ and $(R,\sqrt{r^2-R^2})$.

Fix $R$ sufficiently large.
We estimate the increment of the argument of $f$ along each component of $D_j$.

\smallskip
\noindent
(1) The increment of the argument of exponential sum $f$ along $V_-(r)\cup V_+(r)$ equals
$2\sqrt{r^2-R^2}\,{\rm length}(\Delta_j)$.

\noindent
(2) The increment of the argument along $c_l$ is $O(1)$.

\noindent
(3) By Lemma~\ref{lmFTAexp1}(1), the increment of the argument along the segment
$V(r)$ joining the endpoints of the arc $c_u(r)$ is $O(1)$.

\noindent
(4) Let $S(r)$ be the circular segment bounded by the arc $c_u(r)$ and its chord.
As $r$ increases, the set $S(r)$, up to translation, is monotone decreasing.
Hence, by Lemma~\ref{lmFTAexp1}(2), the number of zeros of $f$ in $S(r)$ is $O(1)$.

\smallskip
\noindent
Combining (1)--(4), we obtain the estimate~\eqref{side}, completing the proof.
\end{proof}
\subsection{Exponential Sums in Several Variables}\label{exp2}
\subsubsection{Definitions and examples}\label{exp21}
We now turn to exponential sums over several variables and to multivariate analogues of Theorem ~\ref{thmexp},
mainly contained in~\cite{Dokl,K84,Few,exp}.
Recall that an exponential sum in $n$ variables is a function on $\C^n$ of the form
\[
f(z)=\sum_{\lambda\in\Lambda} c_\lambda \E^{\lambda(z)},
\]
where $\Lambda$ is a finite subset of the dual space $\C^{n*}$, consisting of complex linear functionals on $\C^n$, and $c_\lambda\in\C$.
The set $\Lambda$ is called the \emph{spectrum} of $f$.
The convex hull $\conv(\Lambda)\subset\Cn*$ is referred to as the \emph{Newton polyhedron} of the exponential sum~$f$.

Given a collection $F=(f_1,\dots,f_k)$ of exponential sums,
we denote by $Z(F)$ the analytic set (or simply the variety) of common zeros of the system
\[
f_1=\cdots=f_k=0.
\]
Since multiplication of each $f_i$ by a nonzero constant does not affect the zero set,
we shall henceforth regard exponential sums as being defined up to a nonzero scalar factor.
In particular, this convention applies to the notations $F$ and $Z(F)$.
\begin{definition}\label{dfX(F)}
Let $F=(f_1,\ldots,f_k)$ be a system of exponential sums in $\C^n$.
We define a current $X(F)$ of bidegree $(k,k)$ by
\[
X(F)(\psi)=\int_{Z(F)}\psi,
\]
where $\psi$ is a compactly supported differential form of bidegree $(n-k,n-k)$, and the integration is taken over the components of $Z(F)$ of dimension $n-k$.
\end{definition}
\begin{corollary}\label{corX(F)}
If $k=n$, then $X(F)$ is a nonnegative measure on $\C^n$, equal to the sum of Dirac measures supported at the isolated points of $Z(F)$.
\end{corollary}
We next recall several constructions required for the definition of the expected current of zeros of a system of random exponential sums; cf.~Definition~\ref{dfExpected1} below.

Let $\Lambda\subset\C^{n*}$ be a finite set.
Denote by $S_\Lambda$ the complex vector space with coordinates $\{u_\lambda\}_{\lambda\in\Lambda}$, and by $S_\Lambda^*$ its dual.
Define the map $\kappa_\Lambda\colon\C^n\to S_\Lambda$ by
\[
\kappa_\Lambda(z)=\bigl\{u_\lambda=\E^{\lambda(z)}\bigr\}_{\lambda\in\Lambda}.
\]
To a functional $\theta=\sum_{\lambda} a_\lambda u_\lambda\in S_\Lambda^*$ we associate the exponential sum
\[
f(z)=\theta\bigl(\kappa_\Lambda(z)\bigr)=\sum_{\lambda} a_\lambda \E^{\lambda(z)}.
\]
Thus, elements of $S_\Lambda^*$ may be identified with exponential sums having spectrum~$\Lambda$.
By construction,
\begin{equation}\label{eqthetaC}
Z(f)=\{z\in\C^n\mid f(z)=0\}=\kappa_\Lambda^{-1}\bigl(\ker\theta\bigr).
\end{equation}
We first describe the averaging procedure in the one-dimensional case.
Let $\P_\Lambda$ denote the projective space associated with $S_\Lambda^*$, i.e., the space of one-dimensional subspaces in $S_\Lambda^*$.
Points of $\P_\Lambda$ may be viewed as exponential sums with spectrum $\Lambda$, defined up to a scalar factor.
Let $\mu_\Lambda$ be the standard Hermitian metric on $\P_\Lambda$, normalized by
\[
\int_{\P_\Lambda} d\mu_\Lambda=1.
\]
The averaged null variety of a random exponential sum with spectrum $\Lambda$ is defined as follows:
\[
\mathfrak U_{1,\Lambda}=\int_{\theta\in\P_\Lambda} Z(\theta)\, d\mu_\Lambda,
\]
where $Z(\theta)=\kappa_\Lambda^{-1}(\ker\theta)$ is considered as the current in $\C^n$.
\begin{definition}\label{dfExpected1}
The \emph{averaged distribution of zeros} of systems of $k$ exponential sums with spectra $\Upsilon=(\Lambda_1,\ldots,\Lambda_k)$ is defined as the current
\[
\mathfrak U_{k,\Upsilon}
=\int_{F\in\P_1\times\cdots\times\P_k} X(F)\, d\mu_1\cdots d\mu_k,
\]
where $\P_i=\P_{\Lambda_i}$ and $\mu_i=\mu_{\Lambda_i}$.
Equivalently, $\mathfrak U_{k,\Upsilon}$ is the expectation of the random current $X(F)$ with respect to the product measure $\mu_1\otimes\cdots\otimes\mu_k$.

In particular, when $k=n$, $\mathfrak U_{n,\Upsilon}$ is a measure on $\C^n$, and for any domain $U\subset\C^n$ the quantity $\mathfrak U_{n,\Upsilon}(U)$ equals the expected number of zeros in $U$ of a system of $n$ exponential sums with spectra $\Lambda_1,\ldots,\Lambda_n$.
\end{definition}
\begin{example}\label{exRe} (See \cite{K84}. A more precise result based on the Fewnomial Theory is given in \cite{Few})
Assume that the spectra $\Lambda_1,\ldots,\Lambda_n$ are contained in $\re(\Cn*)$.
Then, for any domain $D\subset\im(\C^n)$,
\[
\mathfrak U_{n,\Upsilon}\bigl(D+\re(\C^n)\bigr)
=
\frac{n!}{(2\pi)^n}\vol_n(D)\,
\vol_n\bigl(\conv(\Lambda_1),\ldots,\conv(\Lambda_n)\bigr),
\]
where $\vol_n(D)$ and $\vol_n(\conv(\Lambda_1),\ldots,\conv(\Lambda_n))$ denote the $n$-dimensional volume of $D\subset\im(\C^n)$ and the $n$-dimensional
mixed volume of %the corresponding
Newton polyhedra respectively.
\end{example}
\begin{definition}\label{dfZeroExp3}
Let $t\Upsilon=(t\Lambda_1,\ldots,t\Lambda_k)$.
If the limit of currents
\[
\Xi_{k,\Upsilon}
=
\lim_{t\to+\infty}\frac{\mathfrak U_{k,t\Upsilon}}{t^k}
\]
exists, it is called the \emph{asymptotic distribution of zeros} for systems of $k$ exponential sums with spectra $\Upsilon=(\Lambda_1,\ldots,\Lambda_k)$.
\end{definition}
\subsubsection{Main results}\label{exp22}
\begin{theorem}\label{thmMy5}
The asymptotic distribution $\Xi_{k,\Upsilon}$ exists.
\end{theorem}
We next describe $\Xi_{k,\Upsilon}$ explicitly.
The resulting formula depends only on the geometry of the Newton polytopes
$\conv(\Lambda_i)$.
To state it, we recall several standard notions.

By the \emph{Monge--Amp\`ere operator of degree $k$} on a complex manifold we mean the map
\[
(h_1,\ldots,h_k)\longmapsto dd^ch_1\wedge\cdots\wedge dd^ch_k.
\]
 Recall that for a function $g$ on a complex manifold, the $1$-form $d^{c}g$ is defined by
$d^{c}g(v)=dg(iv)$
for any tangent vector $v$, or equivalently,
$
d^{c}= i(\bar\partial-\partial),
$
where $\partial$ and $\bar\partial$ denote the holomorphic and antiholomorphic differentials, respectively.

The values of the Monge--Amp\`ere operator are understood in the sense of currents.
The following statement is usually referred to as the \emph{regularization of the Monge--Amp\`ere operator}.
\begin{assert}\label{reg}
If $h_1,\ldots,h_k$ are continuous plurisubharmonic\footnote{The concept of a plurisubharmonic function on a complex manifold
is a natural analogue of the concept of a convex function in Euclidean space~$\R^n$.
In particular, any convex function on~$\C^n$ is plurisubharmonic.
A detailed exposition of the theory of plurisubharmonic functions and further information can be found, for example, in~\cite{Lel,H1}.}
functions on a complex manifold, then the current
\[
dd^ch_1\wedge\cdots\wedge dd^ch_k
\]
is well defined~\cite{BT}.
More precisely, if the functions $h_i$ are approximated locally uniformly by smooth plurisubharmonic functions, then the corresponding Monge--Amp\`ere currents converge weakly to a limit independent of the chosen approximation.
Moreover, this limit is a \emph{positive current},
in particular, it extends to a continuous functional on compactly supported forms with continuous coefficients,
i.e., it is a so-called \emph{measure-type current};
see \cite{F1}.
\end{assert}
Since any convex function on $\C^n$ is plurisubharmonic, Assertion~\ref{reg} applies in particular to the support functions of the convex polytopes $\conv(\Lambda_i)$.
\begin{theorem}\label{thmMy6}
Let $\Upsilon=(\Lambda_1,\ldots,\Lambda_k)$, and let
$h_1,\ldots,h_k$ be the support functions of the Newton polytopes
$\conv(\Lambda_1),\ldots,\conv(\Lambda_k)$.
Then
\begin{equation}\label{eqMA}
\Xi_{k,\Upsilon}=\frac{1}{(4\pi)^k}dd^ch_1\wedge\cdots\wedge dd^ch_k .
\end{equation}
\end{theorem}
The proofs of Theorems~\ref{thmMy5} and~\ref{thmMy6} rely on an integral-geometric formula of Crofton type, which we now describe; cf.~Theorem~\ref{thmAverage} below.

For $i=1,\ldots,k$, let $E_i$ be a finite-dimensional complex vector space with coordinates
$e_{i,1},e_{i,2},\ldots$.
An element of the dual space $E_i^*$ will be written in the form
\[
w_i=u_{i,1}e_{i,1}+u_{i,2}e_{i,2}+\cdots .
\]
We thus identify $E_i^*$ with a coordinate space equipped with the Hermitian metric induced by the coordinates $u_{i,j}$.

Let $H_i\subset E_i$ be a hyperplane defined by the equation
\[
u_{i,1}e_{i,1}+u_{i,2}e_{i,2}+\cdots=0.
\]
We regard $H_i$ as a point of the projective space $\P_i=\P(E_i^*)$ with homogeneous coordinates $u_{i,1},u_{i,2},\ldots$.
Consider the differential $2$-form
\begin{equation}\label{eqddc}
\omega_i = \frac{1}{4\pi}dd^c \log \sum_j |u_{i,j}|^2 .
\end{equation}
This form is homogeneous of degree zero on $E_i^*\setminus\{0\}$ and therefore descends to a well-defined $(1,1)$-form on $\P_i$.
It is non-degenerate, invariant under the unitary group of $E_i^*$,
and satisfies the normalization property such that the integral of $\omega_i^p$
for all $p\leq\dim \P_i$ over any $p$-dimensional projective subspace $\Q$ in $\P_i$ is equal to~$1$.

We view the form $\omega_1\wedge\cdots\wedge\omega_k$ as a $(k,k)$-form on the product
$\P_1\times\cdots\times\P_k$, identified with the space of $k$-tuples $(H_1,\ldots,H_k)$.

Let $Y$ be a complex submanifold of codimension $p\le k$ in
$E_1\times\cdots\times E_k$.
For a choice of hyperplanes $H_i\subset E_i$, we associate a current
$Y(H_1,\ldots,H_k)$ of bidegree $(k-p,k-p)$ by
\[
Y(H_1,\ldots,H_k)(\phi)
=
\int_{(H_1\times\cdots\times H_k)\cap Y} \phi ,
\]
where $\phi$ is a compactly supported differential form of bidegree $(k-p,k-p)$.

Identifying hyperplanes $H_i$ with points of $\P_i$, we define the averaged current $\Y$ by
\[
\Y(\phi)
=
\int_{(H_1,\ldots,H_k)\in\P_1\times\cdots\times\P_k}
Y(H_1,\ldots,H_k)\wedge\omega_1\wedge\cdots\wedge\omega_k,
\]
where the integrand current $Y(H_1,\ldots,H_k)\wedge\omega_1\wedge\cdots\wedge\omega_k$ of bidegree $(p,p)$ is defined as follows:
$$
(Y(H_1,\ldots,H_k)\wedge\omega_1\wedge\cdots\wedge\omega_k)(\psi)=Y(H_1,\ldots,H_k)(\omega_1\wedge\cdots\wedge\omega_k\wedge\psi)
$$
\begin{theorem}\label{thmAverage}
\[
\Y(\phi)=\int_Y \omega_1\wedge\cdots\wedge\omega_k\wedge\phi .
\]
\end{theorem}
\begin{remark}
Theorem~\ref{thmAverage} is a representative of the class of Crofton-type formulas.
In particular, the classical Crofton formula in a complex projective space $\P$ states that the average number of intersection points of a $k$-dimensional complex submanifold $Y\subset\P$ with a projective subspace of codimension $k$ equals the projective volume of~$Y$; see, for example,~\cite{Sh}.
This result follows from Theorem~\ref{thmAverage} by applying it to the diagonal embedding $Y\hookrightarrow\P^k$.
\end{remark}
\subsubsection{Brief proof of Theorems \ref{thmMy5}, \ref{thmMy6}}\label{exp23}
For the set of spectra $\Upsilon=(\Lambda_1,\ldots,\Lambda_k)$,
consider $k$ vector spaces $E_i$ with coordinates
$$\{u_{\lambda,i}\colon \lambda\in\Lambda_i\}$$
Denote by $\kappa_i\colon\C^n\to E_i$ the mappings
defined as
$$\kappa_i\colon z\mapsto\{u_{\lambda,i}=\E^{\lambda(z)}\},$$
and let $\kappa\colon \C^n\to E_1\times\ldots\times E_k$ be the corresponding diagonal mapping.
We identify the set of k subspaces $H_1,\ldots,H_k$
in $E_1,\ldots,E_k$ with system of exponential sums with spectra $\Lambda_1,\ldots,\Lambda_k$.

Now,
according to Definition \ref{dfExpected1},
applying Theorem \ref{thmAverage} with $Y=\kappa(\C^n)$,
we obtain the following equality
$$
\mathfrak U_k(\Upsilon) = \frac{1}{(4\pi)^k}dd^c\log\sum_{\lambda\in\Lambda_1}\E^{2\re \lambda(z)}\wedge\ldots\wedge dd^c\log\sum_{\lambda\in\Lambda_k}\E^{2\re \lambda(z)}
$$
From this we obtain$$
  \frac{\mathfrak U_{k,t\Upsilon}}{t^k}= \frac{1}{(4\pi)^k}dd^c\left(\frac{1}{t}\log\sum_{\lambda\in\Lambda_1}\E^{2t\re \lambda(z)}\right)\wedge\ldots\wedge dd^c\left(\frac{1}{t}\log\sum_{\lambda\in\Lambda_k}\E^{2t\re \lambda(z)}\right)
$$
%\begin{remark}
%The last equality is an equality of currents  and, therefore, is local.
%Therefore, the assumption that $Y$ is a manifold is not required.
%\end{remark}
%
Now Theorems \ref{thmMy5} and \ref{thmMy6}
follow from the following simple statement,
the derivation of which we omit.
\begin{lemma}
Let $\Lambda$ be a finite set in the space of linear functionals on $\R^N$.
Then, as $t\to+\infty$,
the function
$$
\frac{1}{t}\log\sum_{\lambda\in\Lambda}\E^{2t\lambda(x)}
$$
locally uniformly converges to
twice the support function of the convex polyhedron ${\rm conv}(\Lambda)$.
\end{lemma}
\subsubsection{Concluding remarks}\label{exp24}
We conclude by the description of some aspects of the geometry
connected to the formula (\ref{eqMA}).
\begin{enumerate}
  \item
  In the setting of Example~\ref{exRe},
the support of the measure $\Xi_{n,\Upsilon}$ coincides with the space
$\im(\C^n)$, and one has
\[
\int_{\C^n} \phi\, d\Xi_{n,\Upsilon}
=
\frac{n!}{(2\pi)^n}\vol_n\bigl({\rm conv}(\Delta_1),\ldots,{\rm conv}(\Delta_n)\bigr)
\int_{\im(\C^n)} \phi\, dy,
\]
where $dy$ denotes the standard $n$-dimensional Lebesgue measure on
$\im(\C^n)$, and
$\vol_n({\rm conv}(\Delta_1),\ldots,{\rm conv}(\Delta_n))$
is the mixed volume of the corresponding Newton polytopes.

  \item
  In the one-dimensional case, the support of the measure $\Xi_{1,\Delta}$
is the union of the rays of outward normals to the sides of the Newton
polygon; see~\S\ref{exp1}.
On a ray $r$ dual to a side $\Delta$,
the measure is given by integration over $r$ with weight
$\frac{1}{4\pi}\,{\rm length}(\Delta)$.

  \item
  The support of the current $\Xi_{k,\Upsilon}$ is the $(2n-k)$-dimensional
fan $\mathcal K_\Sigma\subset\C^n$ consisting of cones dual to the
$k$-dimensional faces of the Minkowski sum
\[
\Sigma={\rm conv}(\Lambda_1)+\cdots+{\rm conv}(\Lambda_k)
\]
of Newton polyhedra.
Recall that the cone dual to a face $\Gamma$ of $\Sigma$ consists of all
$z\in\C^n$ such that the linear functional
$\psi(\delta)=\re(z,\delta)$ attains its maximum on $\Sigma$
simultaneously at all points of~$\Gamma$.

  \item
Let $\Upsilon=(\Lambda_1,\ldots,\Lambda_n)$,
where $\Lambda_1=\cdots=\Lambda_n=\Lambda$ and $\Delta=\conv(\Lambda)$.
 We now describe the measure $\Xi_{n,\Upsilon}$ in this special case.
As observed in~(3), the support of this measure is the union of cones
dual to the $n$-dimensional faces of the polyhedron
$\Delta={\rm conv}(\Lambda)$.

Let $K_\Gamma$ be the cone dual to an $n$-dimensional face $\Gamma$ of
$\Delta$.
On $K_\Gamma$, the measure $\Xi_{n,\Upsilon}$ coincides with the
$n$-dimensional Euclidean measure multiplied by the factor
\[
\frac{n!}{(2\pi)^n}\:c(\Gamma)\, \vol_n(\Gamma),
\]
where $\vol_n(\Gamma)$ denotes the Euclidean $n$-volume of $\Gamma$, and
\[
c(\Gamma)=\cos\bigl(T_\Gamma^\perp,\sqrt{-1}\,T_\Gamma\bigr)
\]
is the cosine of the angle between the subspaces
$T_\Gamma^\perp$ and $\sqrt{-1}\,T_\Gamma$ in $\C^{n*}$.
Here $T_\Gamma$ is the tangent space to the face~$\Gamma$ and
$T_\Gamma^\perp$ is its orthogonal complement.
By definition, the cosine of the angle between two $n$-dimensional
subspaces is the absolute value of the Jacobian of the orthogonal
projection from one onto the other with respect to the measure
$\vol_n$.

  \item
Under the assumptions of~(4), let ${\rm pvol}(\Delta)$ denote the value
$\Xi_{n,\Upsilon}(B_1)$, where $B_1\subset\C^n$ is the unit ball centered
at the origin.
We call ${\rm pvol}(\Delta)$ the \emph{pseudovolume} of the polyhedron
$\Delta$.
For example:
\begin{enumerate}
\item
If $\Lambda\subset\re\C^{n*}$, then
${\rm pvol}(\Delta)=\vol_n(\Delta)$;
\item
If $n=1$, then ${\rm pvol}(\Delta)$ coincides with the semiperimeter of
$\Delta$.
\end{enumerate}
The pseudovolume of a convex polyhedron in $\C^n$ may be regarded as a
complex analogue of the usual volume in $\R^n$.
In particular, the map $\Delta\mapsto{\rm pvol}(\Delta)$ is a homogeneous
polynomial of degree~$n$ on the space of convex polyhedra in $\C^n$.

 \item
 It follows from Theorem~\ref{thmMy6} that the pseudovolume admits a
continuous extension, with respect to the Hausdorff topology, to a
polynomial on the space of all compact convex sets.
We now present an equivalent description of this extension in the case
of smooth convex bodies.

Let $K \subset \C^n$ be a smooth convex body, let $x \in \partial K$, and
denote by $\eta(x)$ the outward unit normal vector at $x$.
Define a differential $1$-form $\alpha$ on $\partial K$ by
\[
\alpha(x_t) = \im (x_t, \eta(x)),
\]
where $x_t$ is a tangent vector to $\partial K$ at $x$, and
$(\cdot,\cdot)$ denotes the standard Hermitian inner product on $\C^n$.
Then
\begin{equation}\label{alphaFromdc}
  {\rm pvol}(K) = \int_{\partial K} \alpha \wedge (d\alpha)^{n-1}.
\end{equation}
 Indeed, consider the Gauss map
$G \colon \partial K \to \partial B_1$ defined by $G(x) = \eta(x)$.
A direct computation shows that the pullback of the form $d^c h$,
where $h$ is a support function of $K$,
differs from $\alpha$ by the differential of the function
$\im (x, G(x))$.
Since exact terms do not contribute to the integral in
\eqref{alphaFromdc}, the formula follows.

\item
It is shown in~\cite{Al03} that the pseudovolume defines a translation-
and unitary-invariant valuation on convex bodies in ${\C^n}^*$.
Note that, in contrast to the usual volume, the pseudovolume is not
monotone with respect to inclusion.
In general, the relation $\Delta\subset\Gamma$ does not imply
${\rm pvol}(\Delta)\leq{\rm pvol}(\Gamma)$.
Further results on pseudovolumes of convex bodies may be found
in~\cite{MA,J}.

\item
In a related setting, mixed pseudovolumes arise in estimates of the
asymptotic density of common zeros of entire functions of exponential
growth; see~\cite{growth}.

\end{enumerate}
\end{document}